\documentclass{amsart}

\usepackage[bookmarks,colorlinks,breaklinks]{hyperref}  
\hypersetup{linkcolor=blue,citecolor=blue,filecolor=blue,urlcolor=blue} 

\usepackage{amsmath, amssymb}
\input xy
\xyoption{all}

\newtheorem{innercustomthm}{Theorem}
\newenvironment{customthm}[1]
  {\renewcommand\theinnercustomthm{#1}\innercustomthm}
  {\endinnercustomthm}

\newtheorem{theorem}{Theorem}[section]
\newtheorem{lemma}[theorem]{Lemma}
\newtheorem{corollary}[theorem]{Corollary}
\newtheorem{fact}[theorem]{Fact}
\newtheorem{proposition}[theorem]{Proposition}
\newtheorem{claim}[theorem]{Claim}

\newtheorem{reduction}{Reduction}

\newtheorem*{corollary*}{Corollary}

\theoremstyle{definition}

\newtheorem{remark}[theorem]{Remark}
\newtheorem{definition}[theorem]{Definition}

\def\bG{\operatorname{\mathbb{G}}}
\def\bZ{\operatorname{\mathbb{Z}}}
\def\zf{\operatorname{\mathbb{Z}[F]}}
\def\Fq{\mathbb{F}_q}
\def\Fql{\mathbb{F}_{q^\ell}}

\def\lra{\operatorname{\longrightarrow}}

\def\alg{\operatorname{alg}}
\def\id{\operatorname{id}}

\def\orb{\operatorname{Orb}}
\def\height{\operatorname{ht}}

\begin{document}

\title[Effective isotrivial Mordell-Lang]{Effective isotrivial Mordell-Lang\\ in positive characteristic}

\author{Jason Bell}
\address{Jason Bell\\
University of Waterloo\\
Department of Pure Mathematics\\
200 University Avenue West\\
Waterloo, Ontario \  N2L 3G1\\
Canada}
\email{jpbell@uwaterloo.ca}

\author{Dragos Ghioca}
\address{Dragos Ghioca\\
University of British Columbia\\
Mathematics Department\\
1984 Mathematics Road\\
Vancouver, BC \ V6T 1Z2\\
Canada}
\email{dghioca@math.ubc.ca}

\author{Rahim Moosa}
\address{Rahim Moosa\\
University of Waterloo\\
Department of Pure Mathematics\\
200 University Avenue West\\
Waterloo, Ontario \  N2L 3G1\\
Canada}
\email{rmoosa@uwaterloo.ca}

\subjclass[2010]{14Q20, 14G05, 14G17, 11B85, 11G10}

\thanks{The authors were partially supported by their NSERC Discovery Grants.}

\date{\today}

\begin{abstract}
The isotrivial Mordell-Lang theorem of
~\cite{fsets} describes the set $X\cap\Gamma$ when $X$ is a subvariety of a semiabelian variety $G$ over a finite field~$\Fq$ and $\Gamma$ is a finitely generated subgroup of $G$ that is invariant under the $q$-power Frobenius endomorphism $F$.
That description is here made effective, and extended to arbitrary commutative algebraic groups~$G$ and arbitrary finitely generated $\zf$-submodules $\Gamma$.
The approach is to use finite automata to give a concrete description of $X\cap \Gamma$.
These methods and results have new applications even when specialised to the case when $G$ is an abelian variety over a finite field, $X\subseteq G$ a subvariety defined over a function field~$K$, and $\Gamma=G(K)$.
As an application of the automata-theoretic approach, a dichotomy theorem is established for the growth of the number of points in $X(K)$ of bounded height.
As an application of the effective description of $X\cap\Gamma$, decision procedures are given for the following three diophantine problems:
Is $X(K)$ nonempty? Is it infinite? Does it contain an infinite coset?
\end{abstract}

\maketitle

\setcounter{tocdepth}{1}
\tableofcontents

\section{Introduction}

\noindent
The Mordell-Lang theorem, proved by Faltings~\cite{faltings}, is a celebrated result, which describes the intersection of an algebraic subvariety $X$ of a semiabelian variety $G$, defined over a field of characteristic zero, with a finitely generated subgroup $\Gamma$ of $G$. Faltings' theorem says that this intersection is a finite union of cosets of subgroups of $\Gamma$, which, in particular, illustrates connections between the underlying geometric and algebraic structures on $G$.  

In positive characteristic the naive translation of Faltings' theorem is no longer true. For example, if one takes a smooth curve $X$ of genus at least two defined over a finite field $\mathbb{F}_q$, then $X$ embeds in its Jacobian, $G$.  If one then picks a finitely generated extension $K$ of $\mathbb{F}_q$ such that $X$ has a $K$-point $x$ that is not a $\Fq^{\alg}$-point of $X$ then the orbit of $x$ under the action of the $q$-power Frobenius is infinite in $X$ and lies entirely in $G(K)$, which is a finitely generated subgroup.  But $X(K)$ cannot contain a coset of an infinite subgroup of $G(K)$, since $X$ would then be the Zariski closure of this coset and hence itself an abelian variety, contradicting the fact that it is of genus at least two.

Groundbreaking work of Hrushovski \cite{udi} showed that in a natural sense all counterexamples to the naive translation of the Mordell-Lang theorem to the positive characteristic setting are of this type; namely, they arise from semiabelian varieties over finite fields -- the so-called {\em isotrivial} case.  Indeed,  Hrushovski~\cite{udi} proved a relative function field version of positive characteristic Mordell-Lang in the mid-nineties that treated the isotrivial case as exceptional. Hrushovski, however, did not give a description of what general intersections look like in the isotrivial case and this exceptional case was dealt with in later work of the third author and Scanlon as follows:

\begin{theorem}[Moosa-Scanlon~\cite{fsets}]
\label{iml-semiabelian-intro}
Suppose $G$ is a semiabelian variety over a finite field $\mathbb F_q$ of prime characteristic~$p$, and let $F:G\to G$ be the endomorphism induced by the $q$-power Frobenius.
Suppose $X\subseteq G$ is a closed subvariety defined over a field extension of $\mathbb F_q$, and $\Gamma\leq G$ is a finitely generated $F$-invariant subgroup.
Then $X\cap\Gamma$ is a finite union of sets of the form $S+\Lambda$ where $S\subseteq\Gamma$ is a translate of a sum of $F$-orbits and $\Lambda$ is a subgroup.
\end{theorem}

Here, we are identifying $G$ and $X$ with their points in a sufficiently large algebraically closed field that serves as a universal domain for algebraic geometry in characteristic~$p$.
Also, see Definition~\ref{cfgfset} below for a precise explanation of what we mean by a ``translate of a sum of $F$-orbits", and for a comparison with the exact formulation in~\cite{fsets}.

Many classical diophantine problems can be realised as special cases of the Mordell-Lang theorem.  For example, a special case of the Skolem-Mahler-Lech theorem (see \cite{Everest}) -- asserting that the zero set of a simple linearly recurrent sequence over a characteristic zero field is a finite union of arithmetic progressions along with a finite set -- can be obtained from Faltings' theorem by taking the semiabelian variety to be $\mathbb{G}_m^d$ for some $d\ge 1$ and $\Gamma$ to be an infinite cyclic subgroup.
More generally, theorems on $S$-unit equations can be cast in this framework as well.

In light of these classical diophantine connections, it is very natural to ask whether an effective version of the Mordell-Lang theorem exists.  Questions of effectivity and decidability within the context of diophantine problems enjoy a long history and it is often the case that even basic questions of this nature are very difficult (see, for example, \cite{CMP87}). For example, Skolem's problem (see \cite{OW}), which asks whether one can decide whether an integer-valued linearly recurrent sequence takes the value zero, is still open.  This simple question can be recast in a general Mordell-Lang framework as asking whether one can decide if certain cyclic subgroups of commutative affine algebraic groups intersect certain hypersurfaces non-trivially.  When one goes beyond the cyclic case, questions of this nature are known to be undecidable.  For example, the solution to Hilbert's tenth problem (see \cite{DMR, Mat}) shows that if $\Gamma=\mathbb{Z}^d\subseteq \mathbb{C}^d$ then it is undecidable whether $\Gamma$ intersects a hypersurface non-trivially for $d$ sufficiently large.   On the other hand, in positive characteristic there has been a lot of recent work dealing with decidable phenomena of this type and many positive results have been obtained; see, for example, \cite{derksen, AB12, DM1, DM2, DM3}.

In this paper we show that in the isotrivial case one indeed has an effective version of a general Mordell-Lang theorem.

\begin{customthm}{A}
\label{A}
Let $G$ be a commutative algebraic group defined over a finite field $\mathbb F_q$, let $F:G\to G$ be the $q$-power Frobenius, let $X\subseteq G$ be a closed subvariety defined over a field extension of $\mathbb F_q$, and let $\Gamma\leq G$ be a finitely generated $\zf$-submodule. Given presentations of $G$ and $X$, along with a finite list of generators for $\Gamma$ as a $\zf$-submodule, there is an effective procedure for determining $X\cap\Gamma$.
\end{customthm}

Our approach is not to make effective the proof of Theorem~\ref{iml-semiabelian-intro} from~\cite{fsets}, but rather to establish a new finiteness statement (Propoisition~\ref{bound}, below), effectively, and then use that to explicitly construct, from the given data, a finite automaton that recognises $X\cap\Gamma$.
This is done in Sections~\ref{sect-ewss}--\ref{sect-automaton} below (see Theorem~\ref{effective-iml} and the discussion following it in~$\S$\ref{subsect-effectiveML}).

A word about how the theory of finite automata entered the picture.
Inspired by Dersken's~\cite{derksen} proof of the Skolem-Mahler-Lech theorem in positive characteristic (whose non-effective version is itself a very special case of Theorem~\ref{iml-semiabelian-intro}), the first and third authors developed in~\cite{fsets-SML} a theory of automatic sets that applies to the general isotrivial Mordell-Lang setting.
Recall that, classically, a subset $S$ of the integers is said to be $d$-automatic, for some positive integer $d$, if there is a finite automaton which recognises precisely the set of base-$d$ expansions of the elements of $S$.
In~\cite{fsets-SML} a generalisation is formulated where the integers are replaced by an arbitrary abelian group (e.g., an isotrivial commutative algebraic group), and the positive integer $d$ (or rather multiplication-by-$d$) is replaced by an arbitrary injective endomorphism of the abelian group (e.g., the Frobenius).
This gives rise to a well-behaved theory of {\em $F$-automatic} subsets of isotrivial commutative algebraic groups.
Details on $F$-automatic sets are both reviewed and appropriately generalised (from finitely generated groups to finitely generated $\zf$-modules) in Section~\ref{sect-auto} below.
Theorem~\ref{A} is proved by showing that $X\cap\Gamma$ is effectively $F$-automatic.

As an application of Theorem~\ref{A} we focus on the central concern of the Mordell-Lang problem:
the set of rational points on a subvariety of an abelian variety.

\begin{corollary}
\label{cor1}
Suppose $G$ is an abelian variety over $\mathbb F_q$ and $X\subseteq G$ is a closed subvariety over a function field extension~$K$.
Given presentations of $G,X$, and~$K$, there is an effective procedure for deciding the following problems:
\begin{enumerate}
\item
Is $X(K)$ nonempty?
\item
Is $X(K)$ infinite?
\item
Does $X(K)$ contain a coset of an infinite subgroup of $G$?
\end{enumerate}
\end{corollary}

Such procedures are given in Section~\ref{sect-decide}.
Besides Theorem~\ref{A}, we make use of an algorithm for producing generators for $G(K)$ coming from the finiteness of the Tate-Shafarevich group for abelian varieties over finite fields~\cite{milne68}.

The decision procedure for problem~(3) of Corollary~\ref{cor1} makes use of the dichotomy between sparse and non-sparse regular languages.
That dichotomy yields the following gap theorem for the growth in the number of rational points of bounded height:

\begin{corollary}
Suppose $G$ is an abelian variety over $\mathbb F_q$ and $X\subseteq G$ is a closed subvariety over a function field extension~$K$.
Consider the N\'eron-Tate canonical height on $G$.
Then the number of points in $X(K)$ of height at most $H$ is either bounded above by $C(\log H)^d$ for some positive constants $C$ and $d$, for $H$ sufficiently large, or it is bounded below by $C'\sqrt{H}$ for some positive constant $C'$.
\end{corollary}

This appears as part of a stronger result (Theorem~\ref{thm:equivalence} below).

Beyond simply determining the intersection $X\cap \Gamma$ in Theorem \ref{A}, we are able to give a general structure theorem that comes naturally from a careful analysis of the $F$-automatic sets we produce.   
In particular, we generalise Theorem~\ref{iml-semiabelian-intro} as follows:

\begin{customthm}{B}
\label{B}
Let $G$ be a commutative algebraic group defined over a finite field $\mathbb F_q$, let $F:G\to G$ be the $q$-power Frobenius, let $X\subseteq G$ be a closed subvariety defined over a field extension of $\mathbb F_q$, and let $\Gamma\leq G$ be a finitely generated $\zf$-submodule.
Then $X\cap\Gamma$ is a finite union of sets of the form $S+\Lambda$ where $S\subseteq\Gamma$ is is a translate of a sum of $F$-orbits and $\Lambda=H\cap\Gamma$ for some $H\leq G$ an algebraic subgroup over a finite field.
\end{customthm}

This appears as Theorem~\ref{iml-module} below. The extension from semiabelian varieties to arbitrary commutative algebraic groups is relatively straightforward, and is covered in Section~\ref{sect-ml}.
But letting $\Gamma$ be finitely generated as a $\zf$-module rather than as a group takes more work and requires new methods.
For example, the case of finitely generated $\zf$-submodules of vector groups was dealt with already by the second author in~\cite[Theorem~2.6]{ghioca}, but it does not seem possible to combine this with Theorem~\ref{iml-semiabelian-intro} to deduce Theorem~\ref{B}.
We instead use the fact that $X\cap \Gamma$ is $F$-automatic (this is Corollary~\ref{iml-fauto-module} below) and then carefully analyse the structure of $F$-automatic subsets of isotrivial commutative algebraic groups in Section~\ref{sect-iml=mod}.  
So this does not rely on, but rather recovers and unifies, the results of~\cite{fsets} and~\cite{ghioca}.
In particular, we get an entirely new, automata-theoretic, proof of Theorem~\ref{iml-semiabelian-intro}. 

Although the results from \cite[Theorem~2.6]{ghioca} and Theorem~\ref{iml-semiabelian-intro} are arguably the two most important instances of Theorem \ref{B}, there are nevertheless interesting diophantine problems that involve interactions between the additive and multiplicative structures of fields and are not covered by either result.
For example, a very special case of our Theorems~\ref{A} and \ref{B} that may be of of independent interest is the following. Given a finitely generated $\zf$-submodule $\Gamma_1\subset \mathbb{G}_a$ and a finitely generated subgroup $\Gamma_2\subset \mathbb{G}_m$, we can consider the intersection $\Gamma_1\cap \Gamma_2$ inside the affine line.
Theorem~\ref{B} applies by considering $G:=\mathbb{G}_a\times \mathbb{G}_m$ and intersecting $\Gamma:=\Gamma_1\times\Gamma_2$ with the diagonal $X$ in $G$.
We thus obtain a description of $\Gamma_1\cap \Gamma_2$.
In fact, as the diagonal is an $F$-invariant curve that is not the translate of an algebraic subgroup, we conclude that $\Gamma_1\cap \Gamma_2$ is a finite union of $F$-orbits.
Moreover, using Theorem~\ref{A}, given generators for $\Gamma_1$ and $\Gamma_2$, we obtain an effective algorithm for determining the points $P_1,\dots, P_\ell$ such that $\Gamma_1\cap\Gamma_2$ is the union of the $F$-orbits of $P_1,\dots, P_\ell$.

Finally, let us note that the work of Derksen \cite{derksen} on linear recurrences and much of the work of Derksen and Masser  \cite{DM1, DM2, DM3} on $S$-unit equations in positive characteristic can be recast in a way that is covered by our Theorem~\ref{A} (see \cite{fsets-SML} for details).
On the other hand, the effective result of Adamczewski and the first author \cite[Theorem 4.1]{AB12} is not obviously covered by our work due to the fact that we must work inside a commutative algebraic group; the ineffective version, however, can be deduced from Theorem \ref{B}.

\subsection*{Acknowledgements} We thank Bjorn Poonen and Felipe Voloch for many useful remarks. We are also grateful to the anonymous referee for their useful comments and suggestions, which improved our paper.

\bigskip
\section{Mordell-Lang for finitely generated subgroups of isotrivial commutative algebraic groups}
\label{sect-ml}

\noindent
It does not seem to have been observed before, though it follows rather readily from combining Theorem~\ref{iml-semiabelian-intro} with known structure theorems about algebraic groups, that the conclusions in fact hold for {\em all} commutative algebraic groups over finite fields, and not just semiabelian varieties.
This is a very different situation than the Mordell-Lang theorem in characteristic zero, which fails, for example, in any vector group of dimension greater than $1$ (vector groups do not pose a problem here because in positive characteristic their finitely generated subgroups are all finite). For a more detailed discussion of the Mordell-Lang problem in characteristic $0$ for commutative algebraic groups, we refer the reader to \cite{G-MRL}. 

We record the aforementioned generalisation in this section for the sake of completeness. But first, let us make precise what is meant by ``translates of sums of $F$-orbits".
The following notation and terminology will be used throughout the paper.

\begin{definition}
\label{cfgfset}
Suppose $M$ is an abelian group equipped with an endomorphism $F:M\to M$.
Then by an {\em $F$-orbit} we mean a set of the form
$$S(a;\delta):=\left\{F^{n\delta}a:n<\omega\right\}$$
where $a\in M$ and $\delta$ is a fixed positive integer.
That is, it is the orbit of an element of $M$ under the iterates of a fixed power of $F$.
We will denote the set-sum of $F$-orbits as follows:
$$S(a_1,\dots,a_r;\delta_1,\dots,\delta_r) :=S(a_1;\delta_1)+S(a_2;\delta_2)+\cdots+S(a_r;\delta_r).$$
Finally, we denote by $\mathcal S(M,F)$ the collection of all {\em translates of sums of $F$-orbits}; i.e., of subsets of $M$ of the form $a+S(a_1,\dots,a_r;\delta_1,\dots,\delta_r)$ where $a,a_1,\dots,a_r\in~M$ and $\delta_1,\dots,\delta_r$ are positive integers.
\end{definition}

\begin{remark}
In~\cite{fsets}, finite unions of sets from $\mathcal S(M,F)$ were called ``cycle-free groupless $F$-sets" and the collection of such were denote by $\orb_M$. We will not use this terminology here.
\end{remark}

So, in the conclusion of Theorem~\ref{iml-semiabelian-intro}, when we say that ``$S\subseteq\Gamma$ is a translate of a sum of $F$-orbits" we mean that $S$ is in $\mathcal S(G,F)$.
This description of $X\cap\Gamma$ differs on the face of it from the original in two ways.
First of all, Theorem~7.8 of~\cite{fsets} is stated in terms of ``$F$-cycles" rather than $F$-orbits, but it is explained there (in Lemma~7.1 and the paragraph following the proof of the Theorem~7.8, of that paper) how one can rephrase it in terms of $F$-orbits, and this formulation is more suitable for our purposes.
Secondly, the assumption is made in~\cite[Theorem~7.8]{fsets} that $\Gamma\leq G(K)$ where $K$ is a regular function field extension of $\mathbb F_q$.
But, as explained in~\cite[Remark~7.11]{fsets},  it is not hard to see that one can always attain this situation at the expense of replacing $q$ by $q^r$ and $F$ by $F^r$, for an appropriate choice of $r>0$.
Since every $F^r$-orbit is an $F$-orbit this does not take us out of $\mathcal S(G,F)$.

Here is the promised generalisation to arbitrary commutative algebraic groups.

\begin{theorem}
\label{iml}
Let $G$ be any commutative algebraic group defined over a finite field $\mathbb F_q$, let $F:G\to G$ be the $q$-power Frobenius, let $X\subseteq G$ be a closed subvariety defined over a field extension of $\mathbb F_q$, and let $\Gamma\leq G$ be a finitely generated subgroup which is also invariant under $F$.  
Then $X\cap\Gamma$ is a finite union of sets of the form $S+\Lambda$ where $S\subseteq\Gamma$ is in $\mathcal S(G,F)$ and $\Lambda=H\cap\Gamma$ for some algebraic subgroup $H\leq G$ over a finite field.
\end{theorem}

\begin{proof}
When $G$ is a semiabelian variety this is precisely Theorem~\ref{iml-semiabelian-intro}.
(One takes $H$ to be the Zariski closure of $\Lambda$, which, by rigidity of semiabelian varieties, is an algebraic subgroup over a finite field.) 
It remains therefore to reduce the general case to the case of semiabelian varieties.

First of all, we observe that the desired description of $X\cap\Gamma$ is preserved under images by isogeny.
That is, suppose $G'$ is a commutative algebraic group and $\phi:G'\to G$ is a surjective morphism of algebraic groups with finite kernel, all defined over a finite field.
Supposing the theorem holds of $G'$, we prove it of $G$.
Let $r>0$ be such that $G'$ and $\phi$ are defined over $\mathbb F_{q^r}$, and consider $X':=\phi^{-1}(X)$ and $\Gamma' :=\phi^{-1}(\Gamma)$.
Note that $\Gamma'$ is still a finitely generated group since $\ker\phi$ is finite.
So we have that $\displaystyle X'\cap\Gamma'=\bigcup_{i=1}^\ell S_i'+\Lambda_i'$ where each $S_i'\subseteq\Gamma'$ is in $\mathcal S(G',F^r)$ and each $\Lambda_i'\leq\Gamma'$ is of the form $H_i'\cap\Gamma'$ where $H_i'$ is an algebraic subgroup over $\mathbb F_q^{\alg}$.
Now, one observes that $\phi(S_i')\in\mathcal S(G,F^r)\subseteq\mathcal S(G,F)$ and $\phi(\Lambda_i')=\phi(H_i'\cap\Gamma')=\phi(H_i')\cap\Gamma$, while $\phi(H_i')$ is an algebraic subgroup of $G$ over $\mathbb F_q^{\alg}$.
Hence $\displaystyle X\cap\Gamma=\bigcup_{i=1}^\ell \phi(S_i')+\phi(\Lambda_i')$ has the desired form.

Now, suppose $G$ is a commutative algebraic group over a finite field.
According to~\cite[$\S 5.6$]{brion}, there exists a largest semiabelian subvariety $G_0\subset G$ and a largest connected unipotent algebraic subgroup $U\subset G$, such that $G=U+G_0$ and $U\cap G_0$ is finite.
In particular, the group multiplication map is an isogeny $\phi:U\times G_0\lra G$.
Because of their characteristic properties, both $G_0$ and $U$ are also defined over finite fields.
We may therefore assume that $G=U\times G_0$.
Note that $U$ is of finite exponent; in characteristic $p$ every unipotent commutative algebraic group is a $p$-group.

Since $\Gamma$ is finitely generated, so is its projection on the first factor of $U\times G_0$.
Since $U$ has finite exponent, that projection must be finite.
Letting $\Gamma_0:=\Gamma\cap \left(\{0\}\times G_0\right)$, we obtain that $\Gamma$ is a finite union of cosets of $\Gamma_0$, say,
$\displaystyle \Gamma=\bigcup_{i=1}^\ell \left(h_i + \Gamma_0\right)$.
Therefore
$\displaystyle X\cap\Gamma =\bigcup_{i=1}^\ell \left(h_i + \left((-h_i+X)\cap \Gamma_0\right)\right)$.
Now, for each $i=1,\dots, \ell$, letting $X_i:=(-h_i+X)\cap \left(\{0\}\times G_0\right)$ and applying the semiabelian case to $X_i\cap\Gamma_0$ in $G_0$, we get the desired description for $X\cap\Gamma$.
\end{proof}

\begin{remark}
In the statement of Theorem~\ref{iml}, and indeed throughout this paper, we are implicitly identifying $G$ with the set of its $\mathcal U$-points where $\mathcal U$ is a sufficiently large algebraically closed field that serves as a universal domain for algebraic geometry in characteristic $p$.
In fact, however, the sets $S$ appearing in the conclusion of the theorem can be taken to be in $\mathcal S(G(L),F)$ where $L$ is any algebraically closed field such that $\Gamma\leq G(L)$.
\end{remark}

\begin{remark}
\label{subsec:gen-M-L}
Next, we discuss briefly the case when one drops the requirement that $\Gamma$ is invariant under the Frobenius endomorphism. For arbitrary finitely generated subgroups $\Gamma$ of semiabelian varieties $G$, it is no longer true that the intersection $X(K)\cap\Gamma$ is a finite union of sets of the form $S+H$, where $S\in \mathcal S(G,F)$ and $H$ is a subgroup of $\Gamma$; in fact, the intersection can be quite wild (see \cite[Example~2.3]{G-Y}). The second author was able to prove (see \cite[Theorem~1.9]{G-CMB}) a structure theorem for the intersection $X(K)\cap\Gamma$ when $\Gamma$ is a finitely generated subgroup, no longer $F$-invariant. The building blocks of the structure theorem from \cite{G-CMB} are no longer $F$-orbits, but instead we have sets of the form
$$\left\{a_n\cdot P\colon n\ge 1\right\},$$
where $P$ is some given point in $G$, while $\{a_n\}_{n\ge 1}$ is a linear recurrence sequence of integers with the extra property that the roots of its characteristic equation are distinct algebraic integers of the form $r^m$, where $m$ is a positive integer and $r$ is a root of the equation witnessing the fact that $F$ is integral over $\mathbb{Z}$ inside ${\rm End}(G)$. 
\end{remark}


\bigskip
\section{$F$-automaticity}
\label{sect-auto}
\noindent
Our goal is to give an {\em effective} version of Theorem~\ref{iml}.
Furthermore, we will be able to weaken the assumption that $\Gamma$ is finitely generated as a {\em group} to it being finitely generated as a {\em $\mathbb Z[F]$-module}.
Note that in the original Mordell-Lang context, when $G$ is assumed to be semiabelian, the map $F$ is integral over $\mathbb Z$ and hence every finitely generated $\mathbb Z[F]$-submodule is finitely generated as a group.
This is no longer true for arbitrary commutative algebraic groups, and the generalisation is both natural and significant (see the discussion at the end of the Introduction for an application of the general case).

Our effectivity will come from explicitly describing a finite automaton that recognises the sets $X\cap\Gamma$.
In order to make sense of this we need to review the notion of ``$F$-automaticity" developed by the first and third authors in~\cite{fsets-SML}.
That, as well as the generalisation of the relevant results of~\cite{fsets-SML} from finitely generated $F$-invariant groups to finitely generated $\mathbb Z[F]$-modules, are the goals of this section.

\begin{definition}[Expansions]
\label{defnfexp}
Suppose $M$ is an abelian group, $F:M\to M$ is an injective endomorphism, and $\Sigma\subseteq M$ is finite.
Given a word $w=x_0x_1\cdots x_m\in\Sigma^*$ we set
$$[w]_F:=x_{0}+Fx_{1}+\cdots +F^mx_{m}\ \in\ M$$
and call this the {\em $F$-expansion of $w$}.
Given $\mathcal L\subseteq\Sigma^*$ we denote by $[\mathcal L]_F$ the set of $F$-expansions of the words in $\mathcal L$.
That is,
$[\mathcal L]_F:=\{[w]_F:w\in\mathcal L\}$.
\end{definition}

\begin{definition}[Spanning sets]
\label{defnspan}
Suppose $M$ is an abelian group and $F:M\to M$ is an injective endomorphism.
By an {\em $F$-spanning set for $M$} we will mean a finite subset $\Sigma\subseteq M$ satisfying the following properties:\begin{itemize}
\item[(i)] $[\Sigma^*]_F=M$,
\item[(ii)] $\Sigma$ contains $0$ and is symmetric (i.e., if $x\in\Sigma$ then $-x\in\Sigma$),
\item[(iii)] for all $x_1,\dots,x_5\in \Sigma $ there exist $t,t'\in \Sigma $ such that $x_1+\cdots+x_5=t+Ft'$, and
\item[(iv)] If $x_1,x_2,x_3\in\Sigma$ and $x_1+x_2+x_3\in F(M)$, then there exists $t\in \Sigma$ such that $x_1+x_2+x_3=Ft$.
\end{itemize}
If $\Sigma$ satisfies all but property~(iv) then we will say it is a {\em weak} spanning set.
\end{definition}

\begin{remark}
The above deifnition differs slightly from~\cite[Definition~5.1]{fsets-SML} where injectivity of $F$ was not assumed but an additional condition --  which follows from injectivity together with our~(iv) above -- appears.
\end{remark}

The key property here is~(i) which says that every element of $M$ has an $F$-expansion using $\Sigma$ as ``digits".
Note that we do not ask for this expansion to be unique.
Conditions~(ii) and~(iii) are technically useful and can in practice always be attained by expanding $\Sigma$.
Condition~(iv) is a strong form of ``$F$-purity" of $\Sigma$ in~$M$, and while much of the basics can be done without it, the full development of $F$-automaticity does require it.

\begin{definition}[Automaticity]
Suppose $M$ is an abelian group and $F$ is an injective endomorphism of $M$ such that $M$ admits an $F^r$-spanning set for some $r>0$.
A subset $S\subseteq M$ is defined to be {\em $F$-automatic} if for some $r>0$ and some $F^r$-spanning set $\Sigma$, the set of words $\{w\in\Sigma^*:[w]_{F^r}\in S\}$ is a regular language (see \cite[Chapter 4]{shallit}).
In other words, there is a finite automaton $\mathcal A$ which takes as inputs finite words on the alphabet $\Sigma$, which it reads left to right, such that a word $x_0x_1\cdots x_m$ is accepted by $\mathcal A$ if and only if $x_{0}+F^rx_{1}+\cdots +F^{mr}x_{m}\in S$.
\end{definition}
 
 \begin{remark}
 \label{remfauto}
\begin{itemize}
\item[(a)]
Already in~\cite[Proposition~6.3]{fsets-SML} it is shown that this notion does not depend on the choice of $F^r$-spanning set $\Sigma$.
But in fact, as Hawthorne later observed in~\cite[Proposition~2.6]{hawthorne}, it does not depend on~$r$ either.
That is, if $S$ is $F$-automatic then for any $r>0$ and any $\Sigma$ an $F^r$-spanning set, $\{w\in\Sigma^*:[w]_{F^r}\in S\}$ is regular.
\item[(b)]
If $M$ has an $F^r$-spanning set then it has an $F^{rk}$-spanning set for every $k>0$, this is~\cite[Lemma~5.7]{fsets-SML}.
It follows that a subset is $F$-automatic if and only if it is $F^\ell$-automatic for some, equivalently for all, $\ell>0$.
\end{itemize}
\end{remark}

We can only discuss $F$-automaticity in $M$ if $M$ admits an $F^r$-spanning set for some $r>0$ in the first place.
While this is not always the case, it is shown in~\cite{fsets-SML} to be so in the context of Theorem~\ref{iml}.

\begin{fact}[Bell-Moosa~\cite{fsets-SML}]
\label{spanning-sml}
Suppose $G$ is a commutative algebraic group over~$\mathbb F_q$ and $F:G\to G$ is the endomorphism induced by the $q$-power Frobenius map.
Fix $K$ a function field over $\mathbb F_q$.
Suppose $\Gamma\leq G(K)$ is a subgroup that is preserved by~$F$ and such that $\Gamma/F(\Gamma)$ is finite.
Then $\Gamma$ has an $F^r$-spanning set for some $r>0$.
\end{fact}

This appears as Corollary~5.9 of~\cite{fsets-SML} but for $\Gamma\leq G$ a finitely generated ($F$-invariant) subgroup of $G$.
However, an inspection of the proof given there shows that the only uses of finite-generatedness are to embed $\Gamma\leq G(K)$ for some function field $K$ and to ensure that $\Gamma/F(\Gamma)$ is finite.
(To see that the latter is a consequence of finite-generatedness, note that $F(\Gamma)$ and $\Gamma$ will have the same rank and hence the quotient, being finitely generated, will be finite.)
This justifies our more general formulation, which will be useful when we consider the case of finitely generated $\zf$-modules below.

But first, let us observe that combining Theorem~\ref{iml} with some work in~\cite{fsets-SML}, we obtain an $F$-automaticity result in the context of Mordell-Lang for commutative algebraic groups over finite fields:

\begin{corollary}
\label{iml-auto}
Let $G$ be a commutative algebraic group over a finite field $\mathbb F_q$, let $F:G\to G$ be the $q$-power Frobenius, let $X\subseteq G$ be a closed subvariety defined over a field extension of $\mathbb F_q$, and let $\Gamma\leq G$ be an $F$-invariant finitely generated subgroup.
Then $X\cap\Gamma$ is $F$-automatic in $(\Gamma, F)$.
\end{corollary}

\begin{proof}
This is basically Theorem~\ref{iml} together with~\cite[Theorem~6.9]{fsets-SML}, but some words of explanation are in order.
First of all, by Fact~\ref{spanning-sml}, $\Gamma$ does admit an $F^r$-spanning set for some $r>0$, so the question of $F$-automaticity makes sense.

Theorem~\ref{iml} tells us that $\displaystyle X\cap\Gamma=\bigcup_{i=1}^m S_i+\Lambda_i$ where each $S_i\subseteq\Gamma$ is in $\mathcal S(G,F)$ and each $\Lambda_i\leq\Gamma$ is of the form $H_i\cap\Gamma$ where $H_i$ is an algebraic subgroup over $\mathbb F_q^{\alg}$.
We need to show that these sets are $F$-automatic.
To do so we apply Theorem~6.9 of~\cite{fsets-SML} which says, under precisely the assumptions of this corollary, that the ``$F$-subsets" of $\Gamma$ are $F$-automatic.
So we must first introduce this additional notion from~\cite{fsets}.

Given any abelian group $M$ together with an endomorphism $F:M\to M$, an {\em $F$-cycle} in $M$ is a set of the form
$$C(a,\delta):=\{a+F^\delta(a)+F^{2\delta}(a)+\cdots+F^{n\delta}(a):n<\omega\},$$
where $a\in M$ and $\delta$ is a positive integer.
An {\em $F$-subset} of $M$ is a finite union of sets of the form $C+\Lambda$ where $\Lambda$ is an $F$-invariant subgroup of $\Gamma$ and $C$ is a translate of a finite sum of $F$-cycles.
The connection between $F$-cycles and $F$-orbits is made in~\cite{fsets}.
First of all, every $F$-orbit is a finite union of translates of $F$-cycles; indeed, it is easily checked that $S(a;\delta)=\{a\}\cup\big(a+C(F^\delta(a)-a;\delta)\big)$.
But more is true: a short combinatorial argument given in Lemma~2.9 of~\cite{fsets} shows that if $S\subseteq M$ and $S\in\mathcal S(M',F)$, where $M'$ is a an extension of $M$ to which $F$ extends, then in fact $S$ is a finite union of translates of sums of $F$-cycles in $M$ itself.

The $S_i$ appearing in our description of $X\cap \Gamma$ are subsets of $\Gamma$ that come from $\mathcal S(G,F)$.
So by the above remarks, they are each a finite union of translates of $F$-cycles in $\Gamma$.
Note, however, that the $\Lambda_i$ appearing in $X\cap\Gamma$ need not be $F$-invariant, so we do not necessarily get that $X\cap\Gamma$ is an $F$-subset.

But the $\Lambda_i$ will be $F^\ell$-invariant for some $\ell>0$ because $\Lambda_i=H_i\cap\Gamma$ and $H_i$ is an algebraic subgroup defined over a finite field.
And the $S_i$ are also finite unions of translates of $F^\ell$-cycles.
This is because, in general, as long as $F^\delta-1$ is not a zero divisor in $\mathbb Z[F]\subseteq\operatorname{End}(M)$ for any positive integer $\delta$, then every $F$-cycle in $M$ is a finite unions of translates of $F^\ell$-cycles (this is also done in~\cite{fsets} but see the explanation in~\cite[Fact~2.3]{fsets-SML}).
To observe that $F^\delta-1$ is not a zero divisor in our context see the proof of Theorem~6.9 of~\cite{fsets-SML}.

In conclusion then, $X\cap\Gamma$ is an $F^\ell$-subset of $\Gamma$ for some $\ell>0$.
So we get by ~\cite[Theorem~6.9]{fsets-SML} that $X\cap\Gamma$ is $F^\ell$-automatic in $(\Gamma,F^\ell)$.
By Remark~\ref{remfauto}, it is thus $F$-automatic in $(\Gamma,F)$, as desired. 
\end{proof}

\begin{remark}
The above argument is {\em not} effective.
This is because of its reliance on the isotrivial Mordell-Lang theorem of~\cite{fsets}; we did not construct an automaton that recognises $X\cap\Gamma$.
In Section~\ref{sect-automaton} we will do precisely that.
\end{remark}

As mentioned before, we want to work in the more general setting where $\Gamma$ is a finitely generated $\zf$-submodule of $G$.
But even to make sense of $F$-automaticity in that context, we have to prove that such $\Gamma$ also admit spanning sets.
This is the reason we presented Fact~\ref{spanning-sml} with weaker hypotheses than appear in~\cite{fsets-SML}.
We only need to show that $\Gamma/F(\Gamma)$ is finite.
That is Theorem~\ref{spanning} below, but we need a preliminary proposition.

\begin{proposition}
\label{prop:commutative}
Let $G$ be a commutative algebraic group over $\Fq$, $F:G\to G$ the $q$-power Frobenius, $K$ a finitely generated extension of $\Fq$, and $\Gamma\leq G(K)$ a finitely generated $\bZ[F]$-submodule.
Let $\tilde\Gamma$ be the {\em $F$-pure hull} of $\Gamma$ in $G(K)$. That is,
$$\tilde{\Gamma}:=\{x\in G(K)\colon \text{ there exists }n\geq0\text{ such that }F^n(x)\in \Gamma\}.$$
Then there exists $n_0\geq0$ such that $F^{n_0}(\tilde{\Gamma})\subseteq \Gamma$ for $j=1,2$.
\end{proposition}

\begin{proof}
First of all, observe that if we have an exact sequence
$$\xymatrix{
0\ar[r]& G_1\ar[r]& G\ar[r]^\pi& G_2\ar[r]&0}$$
of commutative algebraic groups over $\Fq$ and the result holds of $G_1$ and $G_2$ then it holds of $G$.
Indeed, let $\Gamma_1:=\Gamma\cap G_1$ and $\Gamma_2:=\pi(\Gamma)$.
Then, for each $j=1,2$, $\Gamma_j$ is a finitely generated $\mathbb Z[F]$-submodule of $G_j(K)$, and so suppose $n_j$ is such that $F^{n_j}(\tilde{\Gamma}_j)\subseteq \Gamma_j$.
One checks readily that $F^{n_1+n_2}(\tilde\Gamma)\subseteq\Gamma$.

Now, by Chevalley's theorem, there is a short exact sequence
$$\xymatrix{
0\ar[r]& L \ar[r]& G\ar[r]^{\pi}& A\ar[r]& 0
}$$
over $\Fq$, where $L$ is a linear algebraic group and $A$ is an abelian variety.
By the previous paragraph we have thus reduced the proposition to the case of abelian varieties and commutative linear algebraic groups.

Consider the case when $G=A$ is an abelian variety over $\Fq$.
The Lang-N\'eron theorem~\cite{lang-neron} tells us that $A(K)$, and hence $\tilde{\Gamma}$, is a finitely generated group.
Fixing a finite set of generators, and letting $n$ be large enough that $F^{n}$ takes each of them into $\Gamma$, we see that $F^{n}(\tilde{\Gamma}_2)\subseteq \Gamma_2$, as desired.

We may therefore assume that $G=L$ is a commutative linear algebraic group over $\Fq$.
Then $G$ admits a decomposition series over $\Fq^{\alg}$ where each quotient is isomorphic to either $\mathbb G_a$ or $\mathbb G_m$ (see~\cite[Proposition~17.38]{milne}).
Working with a power of $F$ if necessary, we may assume this decomposition series is over $\Fq$.
Hence, using short exact sequences as in the first paragraph of this proof (and induction on $\dim G$) it remains to prove the proposition in the cases when $G=\mathbb G_m$ and $G=\mathbb G_a$.

Consider therefore the case of $G=\mathbb G_m$.
We claim that, as in the case of abelian varieties, $\tilde\Gamma$ is a finitely generated group, which, as in that case, will suffice.
Let $S$ be the finite set of places $v$ of $K$ such that the generators of $\Gamma$ are not $v$-adic units, and let $E$ be the $S$-unit group of $K$.
Since $\Gamma\leq E$ and $E$ is $F$-pure in $G(K)$, we have that $\tilde\Gamma\leq E$.
But by~\cite{rosen}, $E$ is a finitely generated group.


Finally, it remains only to consider the case when $G=\bG_a$.
Fix generators $\gamma_1,\dots, \gamma_r$ of $\Gamma$. 
Our strategy is described by the following reduction.

\begin{claim}
\label{claim:sufficient}
It suffices to prove that there exists a finitely generated $\bZ[F]$-module $\Lambda$ satisfying the following properties:
\begin{enumerate}
\item[(i)] $\Gamma\subseteq \Lambda\subseteq \tilde{\Gamma}$; and  
\item[(ii)] for some generators $\lambda_1,\dots, \lambda_s$ for $\Lambda$ as a $\bZ[F]$-module, if $\sum_{i=1}^s c_i\lambda_i=F(x)$ for some $c_i\in\mathbb F_p$ and some $x\in\bG_a(K)$ then $x\in \Lambda$.
\end{enumerate}
\end{claim}

\begin{proof}[Proof of Claim~\ref{claim:sufficient}.]
First of all, condition~(i) yields that $\tilde{\Lambda}=\tilde{\Gamma}$, where $\tilde{\Lambda}$ is (as usually) the $F$-pure hull of $\Lambda$ in $\bG_a(K)$.

Now, condition~(ii) from above yields that $\tilde{\Lambda}=\Lambda$. Indeed, pick  $x\in\bG_a(K)$ and also, let $n\geq 0$ be minimal such that $F^{n}(x)\in \Lambda$; furthermore, assume $n>0$ since otherwise we would have that $x\in \Lambda$, as claimed. Now, since $F^n(x)\in \Lambda$ and the $\lambda_1,\dots, \lambda_s$ generate $\Lambda$ as a $\bZ[F]$ module, there exists some $k\geq 0$ and  $\mathbb F_p$-linear combinations of the $\lambda_j$'s, say $\xi_0,\dots, \xi_k$, such that
$$F^n(x)=\xi_0+F(\xi_1)+\cdots +F^k(\xi_k).$$
But then condition~(ii) above yields that $\xi_0=F(\xi_0')$ (note that we assumed $n>0$) for some $\xi_0'\in \Lambda$. So, actually $F^{n-1}(x)\in \Lambda$, which contradicts the minimality of $n$. Therefore, indeed $n=0$ and so, $\tilde{\Lambda}=\Lambda$.

We have that $\Lambda=\tilde{\Gamma}$. But there exists $n_1\geq 0$ such that $F^{n_1}(\lambda_j)\in\Gamma$ for each $j=1,\dots, s$.
Hence $F^{n_1}(\tilde{\Gamma})\subseteq \Gamma$, as desired.
\end{proof}

Now, if $\lambda_i=\gamma_i$ (for $i=1,\dots, r$) were to satisfy properties~(i)-(ii) from Claim~\ref{claim:sufficient}, then we are done already (and in this case, as previously observed, we would actually have that $\Gamma=\tilde{\Gamma}$). So, we assume that condition~(ii) above is not satisfied by $\lambda_i=\gamma_i$ (since clearly condition~(i) is). This means that there exist some $c_1,\dots, c_r\in\mathbb F_p$ and there exists $\delta\in\bG_a(K)\setminus\Gamma$ such that
$$F(\delta)=\delta^q=c_1\gamma_1+\cdots + \cdots c_r\gamma_r.$$
Let $h(\cdot )$ be the Weil height associated to the function field $K/\Fq$ (note that if $K$ is a finite extension of $\Fq$, then the conclusion we seek would be obvious because then $\bG_a(K)$ is a finite set). The above relation yields that for each place $v$ of the function field $K/\Fq$, we would have that 
$$\max\{1,|\delta|_v\}\le  \max\{1,|\gamma_1|_v,\dots, |\gamma_r|_v\}.$$
Let $i_1\in\{1,\dots, r\}$ such that $c_{i_1}\ne 0$. Then we replace the tuple $(\gamma_1,\dots, \gamma_r)$ by $(\lambda_1,\dots, \lambda_r)$ where each $\lambda_i=\gamma_i$ for $i\ne i_1$, while $\lambda_{i_1}:=\delta$. Clearly, the $\bZ[F]$-module $\Lambda$ spanned by $\lambda_1,\dots, \lambda_r$ satisfies condition~(i) from Claim~\ref{claim:sufficient}. If also condition~(ii) were to be satisfied by the generators $\lambda_1,\dots, \lambda_r$ of $\Lambda$, then we would be done. Otherwise, we proceed as before. However, note that each time when we replace a set  $\lambda_1,\dots, \lambda_r$ by another set  $\eta_1,\dots, \eta_r$ (generating a \emph{larger} $\bZ[F]$-submodule of $\tilde{\Gamma}$), the Weil height of the point 
$$P_\lambda:=[1:\lambda_1:\cdots :\lambda_r]\in \mathbb{P}^r(K)$$
does not increase (note that the Weil height of the above point is computed as $\sum_v\log\max\{1,|\lambda_1|_v,\cdots ,|\lambda_r|_v\}$, after a suitable normalisation of the absolute values $v$ of the function field $K/\Fq$). Since Northcott's Theorem yields the finiteness of the number of points of bounded Weil height from $\mathbb{P}^r(K)$, we conclude that after finitely many steps, we no longer produce new tuples $(\lambda_1,\dots, \lambda_r)$ beyond the tuples already produced in our previous steps. Therefore, at some step, we must have the two conditions~(i)-(ii) from Claim~\ref{claim:sufficient} are satisfied.

This concludes our proof of the Proposition~\ref{prop:commutative}.
\end{proof}

We can now prove that in our context spanning sets exist.

\begin{theorem}
\label{spanning}
Suppose $G$ is a commutative algebraic group over a finite field~$\mathbb F_q$ and $F:G\to G$ is the $q$-power Frobenius.
If $\Gamma\leq G$ is a finitely generated $\bZ[F]$-submodule then $\Gamma/F(\Gamma)$ is finite.
In particular, $\Gamma$ admits an $F^r$-spanning set for some $r>0$.
\end{theorem}

\begin{proof}
Let us deal first with the case when the following properties hold:
\begin{enumerate}
\item
there is an exact sequence 
$$
\xymatrix{
0\ar[r]& L \ar[r] &G\ar[r]^\pi &A\ar[r]& 0
}$$
over $\mathbb F_q$, where $L=U\times \bG_m^t$ for some commutative unipotent algebraic group $U$, and $A$ is an abelian variety; and
\item
there is a function field $K$ over $\mathbb F_q$ such that $\Gamma\leq G(K)$ is {\em $F$-pure} in $G(K)$ in the sense that if $x\in G(K)$ is such that $F(x)\in\Gamma$ then $x\in \Gamma$.
\end{enumerate}
Afterwards we will remove these assumptions.

Note, first of all, that for some $s>0$ we have $[q^s](L(K))\subseteq F(L(K))$.
This is because multiplication by some power of $p$ kills the commutative unipotent algebraic group $U$, and on $\mathbb G_m^t$ multiplication by $q$ agrees with $F$.

We claim that there exists $n>0$ such that $[n](G(K))\subseteq F(G(K))$.
Indeed, let $\Lambda:=\pi(G(K))\leq A(K)$.
By the Lang-Neron theorem, $A(K)$, and hence $\Lambda$, is a finitely generated group.
Hence $\Lambda/F(\Lambda)$ is finite.
Let $m>0$ be such that $m\Lambda\subseteq F(\Lambda)$.
We show that $n:=q^sm$ works.
Suppose $x\in G(K)$.
Then
\begin{eqnarray*}
\pi(mx)
&=&
m\pi(x)\\
&=&F(\lambda)\ \ \text{ for some }\lambda\in\Lambda\\
&=&
F(\pi(y))\ \ \text{ for some }y\in G(K)
\end{eqnarray*}
So $mx-F(y)\in L(K)$.
Hence $q^smx-F(q^sy)=F(z)$ for some $z\in L(K)$, so that $q^smx=F(q^sy+z)\in F(G(K))$, as desired.

Now, by the $F$-purity of $\Gamma$ in $G(K)$, this implies that $n\Gamma\subseteq F(\Gamma)$.
Hence $\Gamma/F(\Gamma)$ is $n$-torsion.
But $\Gamma/F(\Gamma)$ is a finitely generated group as $\Gamma$ is a finitely generated $\mathbb Z[F]$-module.
So $\Gamma/F(\Gamma)$ is finite.

Next, we consider the general case; that is, we drop assumptions~(1) and~(2).
By the structure of commutative linear algebraic groups (see~\cite[Theorem~17.17]{milne}), together with Chevalley's theorem, we know that there is an $\ell\geq1$ such that $G$ does satisfy property~(1) over $\mathbb F_{q^\ell}$.
Let $K$ be a function field extension of $\mathbb F_{q^\ell}$ such that $\Gamma\leq G(K)$.
Note that $\Gamma$ is a finitely generated $\bZ[F^\ell]$-submodule of $G(K)$; one can take $\{F^i(\gamma_j):0\leq i<\ell, 1\leq j\leq k\}$ as generators where $\{\gamma_1,\dots,\gamma_k\}$ generate $\Gamma$ as a $\bZ[F]$-module.
Let $\widetilde\Gamma$ be the $F^\ell$-pure hull of $\Gamma$ in $G(K)$.
By Proposition~\ref{prop:commutative} there is an $n_0\geq 0$ such that $F^{\ell n_0}\widetilde\Gamma\leq\Gamma$.
In particular, $\widetilde\Gamma$ is a finitely generated $\bZ[F^\ell]$-submodule of $G(K)$ in which it is $F^\ell$-pure.
That is, $(G, q^\ell, F^\ell, \widetilde\Gamma, K)$ satisfies properties~(1) and~(2).
Hence, by the first part of the proof, we have that $\widetilde\Gamma/F^\ell(\widetilde\Gamma)$ is finite.
Applying the $\bZ[F^\ell]$-isomorphism $F^{\ell n_0}$, we have that $\widetilde\Gamma/F^{\ell(n_0+1)}(\widetilde\Gamma)$ is finite too.
As $F^{\ell(n_0+1)}\widetilde\Gamma\leq F^{\ell}\Gamma\leq F\Gamma\leq\Gamma\leq\widetilde\Gamma$, we have that $\Gamma/F(\Gamma)$ is finite.

The ``in particular" clause follows by Fact~\ref{spanning-sml}.
\end{proof}

\bigskip
\section{Explicit weak spanning sets}
\label{sect-ewss}

\noindent
Theorem~\ref{spanning} does not explicitly construct an $F^r$-spanning set; indeed, we are not aware of an effective procedure for doing so except in the case of $\Gamma=G(K)$, see~$\S$\ref{subsect-effectivespan} below.
In general, we can, however, explicitly construct $(r,\Sigma)$ where $\Sigma$ is a {\em weak} $F^r$-spanning set.
Recall that this means $\Sigma$ satisfies all but property~(iv) of Definition~\ref{defnspan}.
It will turn out that weak spanning sets suffice for giving an effective description of $X\cap\Gamma$ in the isotrivial Mordell-Lang setting.

We begin with a general construction of a weak spanning set.

\begin{lemma}
\label{explicitwspan}
Suppose $(M,F)$ is an abelian group with an injective endomorphism.
Assume there exist positive integers $r,b$ and integers $b_1,\dots,b_r$, satisfying:
\begin{itemize}
\item[(i)]
$\displaystyle bx=\sum_{i=1}^rb_iF^i(x)$ for all $x\in M$, and
\item[(ii)]
$\displaystyle |b_i|<\frac{b}{6r}$ for all $i=1,\dots,r$.
\end{itemize}
Suppose $M$ is generated as a $\zf$-module by $u_1,\dots,u_s$.
Then
\begin{equation}
\label{sigmaform}
\Sigma:=\left\{\sum_{i=1}^s\sum_{j=0}^{r-1}a_{i,j}F^j(u_i):a_{i,j}\in\mathbb Z, 0\leq|a_{i,j}|<6b\right\}
\end{equation}
is a weak $F^r$-spanning set for $M$.
\end{lemma}

\begin{proof}
Note that every element of $M$ is of the form $\sum_{i=1}^s\sum_{j=0}^{\ell_i}c_{i,j}F^j(u_i)$ and that if each $|c_{i,j}|<6b$ then this has an $F^r$-expansion with digits in $\Sigma$.
So, toward a contradiction, suppose that there is an element
$x=\sum_{i=1}^s\sum_{j=0}^{\ell_i}c_{i,j}F^j(u_i)\in M$
that does not have an $F^r$-expansion with digits in $\Sigma$ and such that $N:=\max_{i,j}|c_{i,j}|$ is least such.
So $N\geq 6b$.
Dividing each $c_{i,j}$ by $b$ we write
$c_{i,j}=bc_{i,j}'+r_{i,j}$ where $|r_{i,j}|<b$.
Then $x= y+z$ where $y=\sum_{i=1}^s\sum_{j=0}^{\ell_i}c_{i,j}'F^j(bu_i)$ and $z=\sum_{i=1}^s\sum_{j=0}^{\ell_i}r_{i,j}F^j(u_i)$.
Using~(i) we get
\begin{eqnarray*}
y
&=&
\sum_{i=1}^s\sum_{j=0}^{\ell_i}\sum_{k=1}^rc_{i,j}'b_kF^{j+k}(u_i)\\
&=&
\sum_{i=1}^s\sum_{\ell=0}^{\ell_i+r}d_{i,\ell}F^j(u_i)
\end{eqnarray*}
where $\displaystyle d_{i,\ell}:=\sum_{j+k=\ell}c_{i,j}'b_k$.
But by~(ii) and the fact that $|c_{i,j}'|\leq \frac{N}{b}$, we have 
$$|d_{i,\ell}|
\leq
r\frac{N}{b}\frac{b}{6r} = \frac{N}{6}.$$
Since the coefficients of $z$ also satisfy $|r_{i,j}|<b\leq\frac{N}{6}$, we have that
$$x=y+z=\sum_{i=1}^s\sum_{j=0}^{L_i}e_{i,j}F^j(u_i)$$
where $e_{i,j}\leq\frac{N}{6}+\frac{N}{6}<N$ contradicting the minimal choice of $N$.

Next we show if $x_1,\dots,x_5\in\Sigma$ then $x:=x_1+\cdots+x_5=t+F^rt'$ for some $t,t'\in\Sigma$.
We can write $x=\sum_{i=1}^s\sum_{j=0}^{r-1}a_{i,j}F^j(u_i)$ with each $|a_{i,j}|<30b$.
As before, we divide by $b$, writing $a_{i,j}=ba_{i,j}'+r_{i,j}$ with $|r_{i,j}|<b$.
Hence $x=y+z$ where $z:=\sum_{i=1}^s\sum_{j=0}^{r-1}r_{i,j}F^j(u_i)$ and
\begin{eqnarray*}
y
&:=&
\sum_{i=1}^s\sum_{j=0}^{r-1}a_{i,j}'F^j(bu_i)\\
&=&
\sum_{i=1}^s\sum_{j=0}^{r-1}\sum_{k=1}^ra_{i,j}'b_kF^{j+k}(u_i)\\
&=&
\sum_{i=1}^s\sum_{\ell=0}^{2r-1}d_{i,\ell}F^j(u_i)
\end{eqnarray*}
where $\displaystyle d_{i,\ell}:=\sum_{j+k=\ell}a_{i,j}'b_k$, and so $|d_{i,\ell}|
<
r30\frac{b}{6r} = 5b$.
We have that
$$x=\sum_{i=1}^s\sum_{j=0}^{r-1}(r_{i,j}+d_{i,j})F^j(u_i)+F^r\left(\sum_{i=1}^s\sum_{j=0}^{r-1}d_{i,j}F^j(u_i)\right)$$
and all coefficients are bounded by $6b$.
So, $t:=\sum_{i=1}^s\sum_{j=0}^{r-1}(r_{i,j}+d_{i,j})F^j(u_i)$ and $t':=\sum_{i=1}^s\sum_{j=0}^{r-1}d_{i,j}F^j(u_i)$ are in $\Sigma$ and $x=t+F^rt'$, as desired.

Finally, it is clear that $\Sigma$ contains $0$ and is symmetric.
Hence $\Sigma$ satisfies all but property~(iv) of Definition~\ref{defnspan}, and is thus a weak $F^r$-spanning set.
\end{proof}

We can now effectively construct weak spanning sets in the isotrivial Mordell-Lang setting:

\begin{proposition}
\label{explcitiwspan-iml}
Suppose $G$ is a commutative algebraic group over~$\mathbb F_q$, presented to us as a Zariski open subset of a Zariski closed subset of~$\mathbb P^n$.
Let $F:G\to G$ be the $q$-power Frobenius.
Then conditions~$(i)$ and~$(ii)$ of Lemma~\ref{explicitwspan} hold of $M=G$ with an effective choice of $b$ and $r$ as a function of $n$.
In particular, if $\Gamma\leq G$ is a $\zf$-submodule generated by $u_1,\dots,u_s$ then formula~$(\ref{sigmaform})$ of that lemma gives an explicit weak $F^r$-spanning set for $\Gamma$.
\end{proposition}

\begin{proof}
By Chevalley's theorem, there is a short exact sequence
$$\xymatrix{
0\ar[r]& L \ar[r]& G\ar[r]^{\pi}& A\ar[r]& 0
}$$
over $\Fq$, where $L$ is a linear algebraic group and $A$ is an abelian variety.
On the other hand, $L=U\times M$ where $M$ is isomorphic to a multiplicative torus over $\Fq^{\alg}$ and $U$ is unipotent; see, for example,~\cite[Theorem~5.3.1]{brion}.

Choose $\ell>0$ sufficiently large so that
\begin{itemize}
\item[(a)]
$M$ is isomorphic to a multiplicative torus over $\mathbb F_{q^\ell}$, and
\item[(b)]
$\displaystyle q^\ell>q^{\frac{\ell}{2}}6\ell(2n+1)4^n$.
\end{itemize}
It is pointed out in Proposition~\ref{subsect-tori} below that this can be done effectively; that $\ell=\max(2^{11} 11!, 2^n n!)$ satisfies~(a).

We show that $b:=q^{\ell(n+1)}$ and $r:=\ell(2n+1)$ satisfy the conditions of Lemma~\ref{explicitwspan}.

Let $a=\dim A$ and $u=\dim U$.
Then $q^u$ annihilates $U$ by unipotency, and hence so does $q^{\ell u}$.
By~(a) we have that $(F^\ell-q^\ell)$ annihilates $M$.
As $A$ is an abelian variety over $\mathbb F_{q^\ell}$ we have that $F^\ell$ on $A$ is the root of a monic integer polynomial of the form
$Q(x):=\sum_{i=0}^{2a}c_ix^i$ where $|c_0|=q^{\ell a}$, $c_{2a}=1$, and
\begin{eqnarray}
\label{ci}
|c_i|\leq{2a\choose i}q^{\ell a-\frac{\ell i}{2}}\ \ \text{ for all }i.
\end{eqnarray}
See \cite[Theorem~1.1,~Chapter~2,~p.~75]{JMilne}.
In any case, on $G$,
\begin{eqnarray*}
0
&=&
q^{\ell u}(F^\ell-q^\ell)Q(F^\ell)\\
&=&
q^{\ell u}\left(\sum_{i=0}^{2a}c_iF^{\ell i+\ell}\right) -q^{\ell u+\ell}\left(\sum_{i=0}^{2a}c_iF^{\ell i}\right).
\end{eqnarray*}
so that
$$q^{\ell(u+1)}c_0
=
\sum_{i=0}^{2a-1}(q^{\ell u}c_i+q^{\ell(u+1)}c_{i+1})F^{\ell(i+1)}\ +\ q^{\ell u}F^{\ell(2a+1)}.$$
Noting that $u+a\leq \dim G\leq n$, we can multiply through by $q^{\ell(n-u-a)}$ to get
\begin{equation}
\label{desired}
q^{\ell(n-a+1)}c_0
=
\sum_{i=0}^{2a-1}(q^{\ell(n-a)}c_i+q^{\ell(n-a+1)}c_{i+1})F^{\ell(i+1)}\ +\ q^{\ell(n-a)}F^{\ell(2a+1)}.
\end{equation}
Note that $|q^{\ell(n-a+1)}c_0|=q^{\ell(n+1)}=b$ and the degree in $F$ on the right hand side is $\ell(2a+1)\leq\ell(2n+1)=r$.
So to show that conditions~(i) and~(ii) of Lemma~\ref{explicitwspan} hold with $b$ and $r$ it suffices to show that all the coefficients on the right hand side of~(\ref{desired}) are strictly bounded in absolute value by $\frac{b}{6r}$.

First consider the coefficient of $F^{\ell(2a+1)}$ in~(\ref{desired}), which is $q^{\ell(n-a)}$.
By condition~(b) on the choice of $\ell$, $q^\ell>6\ell(2n+1)$.
Thus
$$
q^{\ell(n-a)}
=
\frac{q^{\ell(n+1)}}{q^{\ell(a+1)}}
\leq
\frac{q^{\ell(n+1)}}{q^\ell}
<
\frac{q^{\ell(n+1)}}{6\ell(2n+1)}
=
\frac{b}{6r}
$$
as desired.

Next, for each $i=0,\dots, 2a-1$, consider the coefficient of $F^{\ell(i+1)}$.
It satisfies
\begin{eqnarray*}
|q^{\ell(n-a)}c_i+q^{\ell(n-a+1)}c_{i+1}|
&\leq	&
q^{\ell(n-a)}(|c_i|+q^\ell|c_{i+1}|)\\
&\leq&
q^{\ell n}\left({2a\choose i}q^{-\frac{\ell i}{2}}+{2a\choose i+1}q^{\ell-\frac{\ell i+\ell}{2}}\right)\ \ \text{ by~(\ref{ci})}\\
&=&
q^{\ell(n+\frac{1}{2})}\left({2a\choose i}q^{-\frac{\ell i+\ell}{2}}+{2a\choose i+1}q^{-\frac{\ell i}{2}}\right)\\
&\leq&
q^{\ell(n+\frac{1}{2})}\left({2a\choose i}+{2a\choose i+1}\right)\\
&\leq&
\frac{q^{\ell(n+1)}}{6\ell(2n+1)}6\ell(2n+1)q^{-\frac{\ell}{2}}4^n\ \ \text{ as }a\leq n\\
&<&
\frac{b}{6r}\ \ \text{ by choice of $\ell$ satisfying property~(b)}
\end{eqnarray*}
as desired.
\end{proof}

\begin{remark} \label{rem:weakspanningset}
It may be worth extracting the abstract group-theoretic content of the above proof.
Let $(M,F)$ be an abelian group with an injective endomorphism.
Suppose there is a polynomial $P(x)\in \mathbb{Z}[x]$ such that $P(F)$ annihilates~$M$, and such that all the roots of $P$ have modulus at least $\alpha>1$.
Then we can effectively find positive integers~$b$ and~$r$, in terms of $P$ and $\alpha$, which will satisfy the hypotheses of Lemma~\ref{explicitwspan}.
In particular, given a finite set $\Delta\subseteq M$, we can effectively find a weak $F^r$-spanning set for the $\zf$-submodule of $M$ generated by~$\Delta$.
\end{remark}

\begin{proof}
Write $P(x) = C(x-c_1)\cdots (x-c_d)$, where $C$ is a nonzero integer and each $c_i$ is an algebraic number of modulus at least $\alpha$.
Let $B=c_1\cdots c_d$.
So $B$ is a rational number of modulus greater than $1$.
We now pick a positive integer $N$ with the property that $\alpha^N > 6N d 2^d$.
Now, $P(x)$ divides the integer polynomial $Q(x):=C^N (x^N- c_1^N)\cdots (x^N-c_d^N)$, so that $Q(F)$ annhilates $M$ as well.
Observe that the constant coefficient of $Q$ is $C^N B^N$ and every other coefficient is a sum of at most $2^d$ terms that are all at most $|CB|^N/\alpha^N$ in modulus.
Hence, taking $r=Nd$, $b=|CB|^N$, and $b_1,\dots, b_r$ the other coefficients of $Q$ (or their negatives), we have satisfied the hypotheses of Lemma \ref{explicitwspan}.
Formula~$(\ref{sigmaform})$ of that lemma now gives an explicit weak $F^r$-spanning set.
\end{proof}

\bigskip
\section{A finiteness result on Frobenius pullbacks}

\noindent
The automaton we build in the next section to recognising $X\cap \Gamma$ will be based on a certain finiteness result.
We need two bits of notation to state the proposition:

First, given a variety $W$ over $\mathbb F_q$, with $q$ a power of a prime $p$, and a closed subvariety $V\subseteq W$ defined over a field extension of $\mathbb F_q$, and a natural number $\ell$, let us denote by $V^{q^{-\ell}}$ the transform of $V$ by the inverse of the $q$-power Frobenius on~$W$.
So locally, in an affine chart of $W$, this means replacing the coefficients of the defining equations of $V$ with their $q^\ell$ roots.

Secondly, if $(M,F)$ is an abelian group equipped with an injective endomorphism and $\Sigma\subseteq M$ is finite then by $\Sigma(\ell,F)$ we mean the set of elements of $M$ of the form $[w]_F$ where $w\in\Sigma^*$  is a word of length at most~$\ell$. (In~\cite{fsets-SML} this was denoted by the somewhat ambiguous $\Sigma^{(\ell)}$.) 

\begin{proposition}
\label{bound}
Suppose $G$ is a commutative algebraic group over a finite field~$\mathbb F_q$,
$F:G\to G$ is the endomorphism induced by the $q$-power Frobenius map, and $X$ is a closed subvariety of $G$.
Suppose $r>0$ is sufficiently large, $K$ is a finitely generated extension of $\mathbb F_{q^r}$ over which $X$ is defined, and $\Sigma\subseteq G(K)$ is finite.
Consider the following collection of subsets of $G(K)$,
$$\mathcal T_K:=\big\{ (X-\gamma)^{q^{-\ell r}}(K):\ell\geq 0, \gamma\in\Sigma(\ell,F^r)\big\}$$
Then $\mathcal T_K$ is finite.
\end{proposition}

\begin{remark}
As the proof will show, we need $r$ large enough so that $q^r$ is an upper bound on the total degree of a defining set of polynomials over $\Fq$ for the group multiplication on $G$.
\end{remark}

The main technique for proving this proposition comes from~$\S 5$ of~\cite{derksen} where it is called ``Frobenius splitting".
The first point is that while the $(X-\gamma)^{q^{-\ell r}}$ are varieties defined over the (ever increasing) field extensions $K^{q^{-\ell r}}$, the set $(X-\gamma)^{q^{-\ell r}}(K)$ agrees with the $K$-points of a variety defined over $K$, namely the transform of $X-\gamma$ by ``lambda functions".
We now make this precise.

Fix $K$ a finitely generated extension of $\mathbb F_q$.
Since $K$ is of finite degree of imperfection, we can fix a linear basis $1=h_1,\dots,h_m$ for $K$ over $K^{q}$.
We obtain additive operators $\lambda_1,\dots,\lambda_m$ on $K$ with the property that for all $x\in K$,
$$x=\lambda_1(x)^qh_1+\lambda_2(x)^qh_2+\cdots+\lambda_m(x)^qh_{m}.$$

\begin{definition}[Lambda functions]
\label{lambdafunctions}
For $\ell\geq 0$, by an {\em order $\ell$ lambda function} we will mean an $\ell$-fold composition of functions from $\{\lambda_1,\dots,\lambda_m\}$.
We will denote the set of these functions by
$\displaystyle \Lambda(\ell):=\{\lambda_{i_1}\circ\lambda_{i_2}\circ\cdots\circ\lambda_{i_\ell}:\text{ each }1\leq i_j\leq m\}$.
\end{definition}

They have the following multiplicative property.

\begin{lemma}
\label{lambda}
Suppose $P\in K[x]$ is a polynomial in the $n$ variables $x=(x_1,\dots,x_n)$.
Let $\lambda\in\Lambda(\ell)$ and $a\in K^n$.
Then
$$\lambda\big(P(a^{q^\ell})\big)=P^\lambda(a)$$
where $P^\lambda$ denotes the polynomial obtained by applying $\lambda$ to the coefficients of $P$.

In particular,
$P(a^{q^\ell})=0$
if and only if
$P^\lambda(a)=0$ for all $\lambda\in\Lambda(\ell)$.
\end{lemma}

\begin{proof}
Note first of all that $\lambda_i(uv^q)=\lambda_i(u)v$ for all $u,v\in K$ and $i=1,\dots,m$.
Indeed, $u=\lambda_1(u)^qh_1+\cdots+\lambda_m(u)^qh_m$ and so
$$uv^q=\big(\lambda_1(u)v\big)^qh_1+\cdots+\big(\lambda_m(u)v\big)^qh_m.$$
We now prove the $\ell=1$ case of the lemma for monomials $P(x)$ by induction on the total degree. The case of $P$ a constant is clear.
Writing $P(x)=Q(x)x_j$ for a monomial $Q$ we have
$$\lambda_i\big(P(a^q)\big)=\lambda_i\big(Q(a^q)a_j^q\big)=\lambda_i(Q(a^q))a_j=Q^{\lambda_i}(a)a_j=P^{\lambda_i}(a)$$
as desired.
By linearity, the $\ell=1$ case of the lemma follows.
By induction on $\ell$ the general case follows:
$\lambda\circ\lambda_i\big(P(a^{q^{(\ell+1)}})\big)=\lambda\big(P^{\lambda_i}(a^{q^\ell})\big)=P^{\lambda\circ\lambda_i}(a)$.

The left-to-right direction of the ``in particular" clause is an immediate corollary.
For the converse, we need to observe that for $u\in K$, if $\lambda(u)=0$ for all $\lambda\in\Lambda(\ell)$ then $u=0$.
When $\ell=1$ this is clear by choice of $\lambda_1,\dots,\lambda_m$, and the general case follows by induction.
\end{proof}

\begin{corollary}
\label{prolong}
Suppose $V\subset\mathbb A^n$ is a variety defined by the vanishing of polynomials $P_1,\dots, P_s\in K[x_1,\dots,x_n]$.
Then 
$$\displaystyle V^{q^{-\ell}}(K)=\{a\in K^n:P_i^{\lambda}(a)=0,i=1,\dots,s, \lambda\in\Lambda(\ell)\}.$$
\end{corollary}

\begin{proof}
This is immediate from Lemma~\ref{lambda}.
\end{proof}

Our proof of Proposition~\ref{bound} will take the following form: Working in an affine chart we find bounds, independently of $\ell$, on the degrees and ``heights" of the coefficients of the polynomials $P_1^\lambda,\dots,P_s^\lambda$ as $\lambda$ ranges in $\Lambda(\ell)$ and $\operatorname{Zeros}\{P_1,\dots,P_s\}=X-\gamma$ ranges among the translates of $X$ by $\gamma\in\Sigma(\ell,F^r)$.
Since, by Corollary~\ref{prolong}, $(X-\gamma)^{q^{-\ell}}(K)$ is the set of $K$-points of the variety defined by these $P_i^\lambda$, the degree and height bounds on the $P_i^\lambda$ will imply that there are only finitely many possibilities for $(X-\gamma)^{q^{-\ell}}(K)$.

We now describe what our naive notion of ``height" will be.
We do not quite use the usual canonical height because we are interested in keeping everything effective -- see~$\S$\ref{subsect-effectivity} below.
Let $W\subseteq K$ be a finite dimensional  (and hence finite) $\mathbb F_q$-subspace such that $K=\mathbb F_q(W)$.
We set $W^0:=\mathbb F_q$ and for $\ell>0$ denote by $W^\ell$ the $\mathbb F_q$-span of the $\ell$-fold products of elements in $W$.
So $\displaystyle \mathbb F_q[W]=\bigcup_{\ell<\omega}W^\ell$.

\begin{definition}[Height]
\label{ht0}
Suppose $W\subseteq K$ is a finite dimensional $\mathbb F_q$-subspace such that $K=\mathbb F_q(W)$.
Define  $\height_W:\mathbb F_q[W]\to\mathbb N$ by $\height_W(x):=\min\{\ell:x\in W^\ell\}$.
We can extend this to $K$ by defining $\height_W(u)$ to be the infimum of $\max\big\{\height_W(x),\height_W(y)\big\}$ where $u=\frac{x}{y}$ and $x,y\in\mathbb F_q[W]$.
More generally, given $a\in \mathbb{P}^n(K)$, we define $\height_W(a)$ to be $\inf\max\big\{\height_W(x_0),\ldots ,\height_W(x_n)\big\}$, where the infimum is taken over all representations $a=[x_0:x_1:\cdots:x_n]$ with each $x_i\in \mathbb F_q[W]$.
\end{definition}

Note that this height satisfies the Northcott property: the set of points of height bounded by $N$ is finite because $W^N$ itself is a finite set.
In order for this height function to have the further properties that we desire, we need to choose $W$ carefully.
We follow~$\S 5$ of~\cite{derksen} closely here.
The following, for example, appears in the proof of Proposition~5.2 of~\cite{derksen}, and we repeat it here for the sake of completeness.

\begin{lemma}
\label{R}
There exists a finitely generated $\mathbb F_q$-subalgebra $R$ of $K$, containing $h_1,\dots,h_m$ as well as the generators of $K$ over $\mathbb F_q$, such that $\lambda_i(R)\subseteq R$ for all $i=1,\dots,m$.
\end{lemma}

\begin{proof}
Let $R$ be the $\mathbb F_q$-algebra generated by $h_1,\dots,h_m$ along with the generators of $K$ over $\mathbb F_q$.
So $R^qh_1 +\cdots +R^qh_m\subseteq R$, and we need to modify $R$ so as to get equality.
So consider the $R^q$-module $R/(R^qh_1 +\cdots R^q h_m)$.

We first claim that it is torsion.
Indeed, if $a\in R$ then
$$a=\lambda_1(a)^qh_1+\lambda_2(a)^qh_2+\cdots+\lambda_m(a)^qh_{m},$$ and so if we let a nonzero $b\in R^q$ be such that $b\lambda_i(a)^q\in R^q$ for each $i=1,\dots, m$, then  $ba\in R^qh_1 +\cdots +R^q h_m$.

On the other hand, if $t_1,\dots, t_\nu$ generate $R$ as an $\mathbb F_q$-algebra, then as an $R^q$-module it is generated by the finite set $\{t_1^{r_1}\cdots t_\nu^{r_\nu}:\text{ each }0\leq r_i<q\}$.
So the quotient $R/(R^qh_1 +\cdots R^q h_m)$ is a finitely generated $R^q$-module.

Hence there is nonzero $g\in R$ such that $g^qR\subseteq R^qh_1 +\cdots R^q h_m$.
Localising at $g$, it is not hard to see that 
$R_g=R_{g^q}=R_g^qh_1 +\cdots R_g^qh_m$,
and so~$R_g$ works.
\end{proof}

Fix $R$ as in the lemma and let $W$ be a finite dimensional $\Fq$-subspace of $R$ containing  $h_1,\dots,h_m$ as well as generators for $R$ over $\Fq$.
We work with $\height_W$.
The following is the key property of $\height_W$, and it appears in the proof of Proposition~5.2 of~\cite{derksen}.
We include it here, again for completeness.

\begin{lemma}
\label{lambdaW}
There is a constant $D\geq 0$ such that for all $i=1,\dots,m$, and all $a\in R$, we have
$\height_W\big(\lambda_i(a)\big)\leq \lfloor\frac{\height_W(a)}{q}\rfloor +D$.
\end{lemma}

\begin{proof}
For any set $A\subseteq K$, let $A^{\langle q\rangle}:=\{a^q:a\in A\}$.
Note that when $A$ is a subring of $K$ we have just been denoting this by $A^q$, but when working with $\mathbb F_q$-linear subspaces $V\subseteq K$ it is worth being explicit so as to distinguish between the subspace $V^{\langle q\rangle}$ of $q$-powers and the subspace $V^\ell$ generated by the $\ell$-fold products.

Now, it suffices to show that for all $\ell\geq 0$, 
$\displaystyle \lambda_i(W^\ell)\subseteq W^{\lfloor\frac{\ell}{q}\rfloor +D}$.
That is, we need to prove that
$$W^\ell\subseteq (W^{\lfloor\frac{\ell}{q}\rfloor +D})^{\langle q\rangle}h_1+\cdots+(W^{\lfloor\frac{\ell}{q}\rfloor +D})^{\langle q\rangle}h_m$$
for all $\ell\geq 0$.

Let $d:=\dim W$ and fix a basis $e_1,\dots,e_d$ of $W$ over $\mathbb F_q$.
Then $W^\ell$ is spanned by elements like $e_1^{n_1}\cdots e_d^{n_d}$ where $n_1+\cdots+n_d=\ell$.
Writing each $n_i=qm_i+r_i$ where $0\leq r_i<q$, we have that
$$e_1^{n_1}\cdots e_d^{n_d}=(e_1^{m_1}\cdots e_d^{m_d})^q(e_1^{r_1}\cdots e_d^{r_d}).$$
Note that $m_1+\cdots m_d\leq\lfloor\frac{\ell}{q}\rfloor$ while $r_1+\cdots+r_d\leq d(p-1)$.
So
$$W^\ell\subseteq (W^{\lfloor\frac{\ell}{q}\rfloor})^{\langle q\rangle} \cdot W^{d(q-1)}.$$
Here for vector subspaces $U,V$, by $U\cdot V$ we mean the vector subspace spanned by the product $uv$ where $u\in U$ and $v\in V$.

On the other hand, from Lemma~\ref{R} we have that 
$R=R^qh_1 +\cdots +R^qh_m$.
Since $\displaystyle R=\bigcup_{\ell<\omega} W^\ell$, the finite set $W^{d(q-1)}$ must be contained in
$$(W^D)^{\langle q\rangle}h_1 +\cdots +(W^D)^{\langle q\rangle} h_m$$
for some $D\geq 0$.
Putting these together we get
$$W^\ell\subseteq (W^{\lfloor\frac{\ell}{q}\rfloor +D})^{\langle q\rangle}h_1+\cdots+(W^{\lfloor\frac{\ell}{q}\rfloor +D})^{\langle q\rangle}h_m$$
as desired.
\end{proof}

Now, we have a commutative algebraic group $G$ over $\mathbb F_q$.
Write $G$ as a Zariski open subset of a Zariski closed subset of $\mathbb P^n$ over $\mathbb F_q$.
So $\height_W$ is defined on $G(K)$.
That it is compatible with the algebraic group structure is the following.

\begin{lemma}
\label{height+}
Let $C_0$ be an upper bound on the total degree of a defining set of polynomials over $\Fq$ for the group multiplication on $G$.
Then, for all $a,b\in G(K)$,
$$\height_W(a+b)\leq C_0\max\{\height_W(a),\height_W(b)\}.$$
\end{lemma}

\begin{proof}
By definition, for all $x,y\in R$, $\height_W(x+y)\le \max\{\height_W(x),\height_W(y)\}$ and $\height_W(xy)\le \height_W(x)+\height_W(y)$.
As the coefficients of the polynomials defining the group multiplication are in $\mathbb F_q$ and hence have height $0$, a straightforward computation shows that $\height_W(a+b)\leq C_0\max\{\height_W(a),\height_W(b)\}$.
\end{proof}

Suppose now that we have a finite set $\Sigma\subseteq G(K)$.
Expanding $W$ if necessary, we may and will assume that every element of $\Sigma$ has a representation with all co-ordinates in $W$.

\begin{corollary}
\label{htsigma}
Let $C_0$ be as in Lemma~\ref{height+} and assume that $q\geq C_0$.
Then, for all $\ell\geq 1$, and all $\gamma\in \Sigma(\ell,F)$, we have $\height_W(\gamma)\leq C_0q^\ell$.
\end{corollary}

\begin{proof}
We prove this by induction on $\ell$, the case of $\ell=1$ being our assumption that every element of $\Sigma$ has height one.
Suppose $\gamma\in\Sigma(\ell+1,F)$.
So
$$\gamma=x_0+Fx_1+\cdots+F^{\ell}x_{\ell}$$
where $x_0,\dots,x_\ell\in\Sigma$.
Writing $x_\ell=[a_0:\cdots:a_n]$ with each $a_i\in W$, we have that $F^\ell x_\ell=[a_0^{q^\ell}:\cdots:a_n^{q^\ell}]$ and $a_i^{q^\ell}\in W^{\langle q^\ell\rangle}\subseteq W^{q^\ell}$.
Hence, $\height_W(F^\ell x_\ell)\leq q^\ell$.
So
\begin{eqnarray*}
\height_W(\gamma)
&\leq&
C_0\max\{\height_W(x_0+\cdots+F^{\ell-1}x_{\ell-1}),\height_W(F^\ell x_\ell)\}\ \ \text{ by Lemma~\ref{height+}}\\
&\leq&
C_0\max\{C_0q^{\ell-1},q^\ell\}\ \ \text{ by induction}\\
&=&
C_0q^\ell \ \ \ \text{ as }q\geq C_0
\end{eqnarray*}
as desired.
\end{proof}

We are now ready to prove the proposition.

\begin{proof}[Proof of Proposition~\ref{bound}]
Write $G$ as a Zariski open subset of a Zariski closed subset $\overline G\subseteq \mathbb P^n$ over $\mathbb F_q$.
Fix a finite open cover $\{U_i:i=1,\dots,s\}$ of $G$, and for each $i,j\in \{1,\ldots ,s\}$ polynomials $P_0^{(i,j)},P_1^{(i,j)},\ldots ,P_n^{(i,j)}\in \mathbb F_q[x_0,\ldots ,x_n;y_0,\ldots,y_n]$, homogeneous in $(x_0,\dots,x_n)$ and in $(y_0,\dots, y_n)$, with no common zeros on $U_i\times U_j$, such that
for $(p,q)\in U_i\times U_j\subseteq\mathbb P^n\times\mathbb P^n$, we have 
$$p+q=[P_0^{(i,j)}(p,q):P_1^{(i,j)}(p,q):\cdots :P_n^{(i,j)}(p,q)].$$
Let $C_0$ be the maximum of the degrees of the $P_k^{(i,j)}$s.

Suppose $r>0$ is such that $q^r\geq C_0$ and fix a function field extension $K$ of $\mathbb F_{q^r}$ such that $X$ is over $K$, as well as a finite set $\Sigma\subseteq G(K)$.
Fix $h_1,\dots,h_m$ an $\mathbb F_{q^r}$-linear basis for $K$ over $K^{q^r}$, and denote by $\lambda_1,\dots,\lambda_m$ the corresponding (order~$1$) lambda functions.
We will be applying the above lemmas, but with $q^r$ replacing $q$ and $F^r$ replacing~$F$.

Using Lemma~\ref{R}, fix a finitely generated $\mathbb F_{q^r}$-subalgebra $R$ of $K$, such that
\begin{itemize}
\item
$h_1,\dots,h_m\in R$,
\item
$\operatorname{Frac}(R)=K$,
\item
$\lambda_i(R)\subseteq R$ for all $i=1,\dots,m$, and
\item
every element of $\Sigma$ has a representation with all co-ordinates in $R$.
 \end{itemize}
 Next, fix a finite dimensional $\mathbb F_{q^r}$-vector subspace $W$ of $R$ that contains $h_1,\dots,h_m$, as well as generators for $R$, and such that every elements of $\Sigma$ has a representation with all co-ordinates in $W$.
Let $\height_W$ be the corresponding height function on $R$ and on $\mathbb P^n(K)$  studied above.

Let $\overline X$ be the Zariski closure of $X$ in $\overline G$ so that $X=\overline X\cap G$, and $\overline X$ is given by homogeneous polynomials say $Q_1,\dots,Q_t\in K[x_0,\dots,x_n]$.

Now fix $\ell\geq 0$, and $\gamma\in\Sigma(\ell,F^r)$.
Suppose $\gamma\in U_j$.
Fixing $p\in U_i$ we have that $p\in X-\gamma$ if and only if $p+\gamma\in X$, that is, if and only if
$$Q_\nu\big(P_0^{(i,j)}(p,\gamma),\dots,P_n^{(i,j)}(p,\gamma)\big)=0$$
for all $\nu=1,\dots,t$.
That is, $(X-\gamma)\cap U_i$ is defined by the vanishing of
$$Q_{\nu,\gamma,i,j}(x):=Q_\nu\big(P_0^{(i,j)}(x,\gamma),\dots,P_n^{(i,j)}(x,\gamma)\big)$$
for $\nu=1,\dots,t$.
It follows by Lemma~\ref{lambda} that
$$(X-\gamma)^{q^{-\ell r}}(K)\cap U_i=\{x\in U_i(K):Q_{\nu,\gamma,i,j}^\lambda(x)=0\text{ for all }\nu=1,\dots,t,\lambda\in\Lambda(\ell)\}.$$

Note that the total degrees of the $Q_{\nu,\gamma,i,j}^\lambda$ are bounded independently of $\ell$ and~$\gamma$.
Indeed, if $C_1$ is the maximum of the total degrees of $Q_1,\dots,Q_t$, then $C_1C_0$ is such a bound.
So, in order to show that $\mathcal T_K$ is finite, it suffices to prove that there is a height bound for the coefficients of $Q_{\nu,\gamma,i,j}^\lambda$ that is independent of $\ell$ and $\gamma$.
(Note that $\nu, i, j$ range over finite sets.)
And in order to give a bound on $|\mathcal T_K|$ it suffices to give a bound on the height of these coefficients.

By Corollary~\ref{htsigma}, $\height_W (\gamma)\leq C_0q^{\ell r}$.
(This is where we use that $q^r\geq C_0$.)
The coefficients of $P_k^{(i,j)}(x,\gamma)$ therefore have $\height_W$ bounded by $C_0^2q^{\ell r}$.
If $C_2$ is the maximum of the heights of the coefficients of the $Q_\nu$, then we get that the coefficients of $Q_{\nu,\gamma,i,j}(x)$ have $\height_W$ bounded by $C_1C_0^2q^{\ell r}+C_2$. 
Next we analyze what happens when we apply $\lambda\in\Lambda(\ell)$ to these coefficients. 
Fixing $a\in R$ a coefficient of $Q_{\nu,\gamma,i,j}(x)$, we compute the heights of $\lambda_1(a),\dots,\lambda_m(a)$.
Since $\height_W(a)\leq C_1C_0^2q^{\ell r}+C_2$, letting $D\geq 0$ be is as in Lemma~\ref{lambdaW} but applied to $q^r$, that lemma implies $\height_W\big(\lambda_k(a)\big)\leq C_1C_0^2q^{(\ell-1) r}+\frac{C_2}{q^r}+D$ for all $k=1,\dots,m$.
Iterating $\ell$ times, we have that for all $\lambda\in\Lambda(\ell)$,
$$\height_W\big(\lambda(a)\big)\leq C_1C_0^2+\frac{C_2}{q^{\ell r}}+\frac{D}{q^{(\ell-1)r}}+\frac{D}{q^{(\ell-2)r}}+\cdots+D\leq C_1C_0^2+C_2+\frac{Dq^r}{q^r-1}$$
where the final inequality is by geometric series.
So the coefficients of all the $Q_{\nu,\gamma,i,j}^\lambda$ have height bounded by $C_1C_0^2+C_2+\frac{Dq^r}{q^r-1}$.
\end{proof}

\subsection{Proposition~\ref{bound} is effective}
\label{subsect-effectivity}
That is, a bound on $|\mathcal T_K|$ can be effectively determined from $r$ and $\Sigma$, along with defining equations for $G$ and $X$ as well as a finite presentation of~$K$.
This is more or less clear from the proof we have given, but we now summarise the justification for  this effectivity claim.

First of all, given $K=\operatorname{Frac}(S)$ with a finite presentation of $S:=\Fq[T_1,\dots,T_s]/I$, we can effectively determine a $K^q$-basis for $K$.
This is argued in Appendix~\ref{subsect-compK} below.
Let $1=h_1,\dots,h_m$ be such a basis.

Next, the construction of $R$ in Lemma~\ref{R} is effective.
Indeed, letting $c_1,\dots,c_s$ be the given generators for $S$ -- namely the indeterminates $T_i$ modulo $I$ -- the proof of that lemma effectively finds $g\in \Fq[h_1,\dots,h_m,c_1,\dots,c_s]$ such that $$R:=\Fq\left[h_1,\dots,h_m,c_1,\dots,c_s,\frac{1}{g}\right]$$ satisfies the lemma.

We thus have $R=\Fq[W]$ where $W$ is the $\mathbb F_q$-span of $h_1,\dots,h_m,c_1,\dots,c_s,\frac{1}{g}$.
We can work effectively with $\height_W$ because we can decide whether two elements of $R$ are the same.
Indeed, we can find $h\in R$ such that $R\subseteq S[\frac{1}{h}]$ and from the given finite presentation of $S$ we obtain a finite presentation of $S[\frac{1}{h}]$, so that deciding equality reduces to an ideal membership problem, which is testable by Gr\"obner bases.

Finally, the constant $D$ found in Lemma~\ref{lambdaW} can, by the proof of that lemma, be taken to be $\max\{\operatorname{ht}_W(v):v\in W^{\dim W(q-1)}\}$.
So it too is effectively determined. 

With these effective ingredients the proof we have given of Proposition~\ref{bound} gives explicit bounds on the total degrees and $\height_W$ of the coefficients of the polynomials that arise in the description of the elements of $\mathcal T_K$, thus yielding an effective bound on $|\mathcal T_K|$.\qed

\bigskip
\section{An automaton recognising $X\cap\Gamma$}
\label{sect-automaton}

\noindent
Fix
$G$ a commutative algebraic group over $\mathbb F_q$,
$F:G\to G$ the $q$-power Frobenius endomorphism,
$\Gamma\leq G$ a finitely generated $\zf$-submodule, and
$X\subseteq G$ is a closed subvariety defined over some field extension of $\mathbb F_q$.
By Proposition~\ref{explcitiwspan-iml} we have $r>0$ and $\Sigma\subseteq\Gamma$ a weak $F^r$-spanning set for $\Gamma$.
(We know by Theorem~\ref{spanning} that we can in fact choose $r$ and $\Sigma$ so that $\Sigma$ is an actual an $F^r$-spanning set, but we do not use this in the construction of our automaton, and the advantage of asking only for weak spanning sets is that $(r,\Sigma)$ can be effectively constructed.)

By the proof of~\cite[Lemma~5.7]{fsets-SML}, for every $m>0$, $\Sigma(m,F^r)$ is a weak $F^{rm}$-spanning set for $\Gamma$.
We may therefore assume that $r$ is sufficiently large so that there are defining polynomials over $\Fq$ for the multiplication on $G$ of degree less than $q^r$.
Hence Proposition~\ref{bound} applies.

Fix also a finitely generated extension $K$ of $\mathbb F_{q^r}$ such that $\Gamma\leq G(K)$ and $X$ is defined over $K$.

We wish to describe a finite automaton $\mathcal A$ on the alphabet $\Sigma$ such that $w\in\Sigma^*$ is accepted by $\mathcal A$ if and only if $[w]_{F^r}\in X\cap\Gamma$.
The set of states of $\mathcal A$ will be the $\mathcal T_K$ of Proposition~\ref{bound}.
The initial state will be $X(K)$, which corresponds to $\gamma=0$ and $\ell=0$.
The accepting states are those sets $(X-\gamma)^{q^{-\ell r}}(K)\in\mathcal T_K$ which contain~$0$.
Here's the transition rule: if the machine is in state $(X-\gamma)^{q^{-\ell r}}(K)$ and reads the letter $x$ then it should move to state $(X-\gamma -F^{\ell r} x)^{q^{-(\ell+1)r}}(K)$.
Note that $\gamma+F^{\ell r} x\in\Sigma(\ell+1,F^r)$, so that this is indeed a state of $\mathcal A$.
The following shows that the rule is well-defined.

\begin{lemma}
For any $x\in G(K)$, subvarieties $V$ and $W$ of $G$ over $K$, and $\ell,\ell'\geq 0$, if $V^{q^{-\ell}}(K)=W^{q^{-\ell'}}(K)$ then $(V-F^\ell x)^{q^{-\ell-1}}(K)=(W-F^{\ell'}x)^{q^{-\ell'-1}}(K)$.
\end{lemma}

\begin{proof}
Fix $a\in G(K)$.
Then
\begin{eqnarray*}
a\in (V-F^\ell x)^{q^{-\ell-1}}
&\iff&
F^{\ell+1}a\in V-F^\ell x \ \ \ \\
&\iff&
F^\ell(Fa+x)\in V\\
&\iff&
Fa+x\in V^{q^{-\ell}}\\
&\iff&
Fa+x\in W^{q^{-\ell'}} \ \text{ by assumption and as $Fa+x\in G(K)$}\\
&\iff& 
F^{\ell'}(Fa+x)\in W\\
&\iff&
F^{\ell'+1}a\in W-F^{\ell'} x\\
&\iff&
a\in (W-F^{\ell'} x)^{q^{-\ell'-1}}
\end{eqnarray*}
as desired.
\end{proof}

So the transition function is well-defined, and we have a finite automaton.
It remains to verify that it does what we want.

\begin{lemma}
\label{l=la}
The automaton $\mathcal A$ accepts exactly those words $w\in \Sigma^*$ such that $[w]_{F^r}\in X\cap\Gamma$.
\end{lemma}

\begin{proof}
Suppose that $w=x_0x_1\cdots x_{\ell-1}$ for some $\ell\geq 0$.
Then
\begin{eqnarray*}
w\text{ is accepted}
&\iff&
0\in (X-x_0-F^rx_1-\cdots-F^{(\ell-1)r} x_{\ell-1})^{q^{-\ell r}}(K)\\
&\iff&
0\in X-x_0-F^rx_1-\cdots-F^{(\ell-1)r} x_{\ell-1}\ \ \text{ as $0\in G(\mathbb F_q)$}\\
&\iff&
[w]_{F^r}\in X.
\end{eqnarray*}
Since $[w]_{F^r}\in\Gamma$ always, this is as desired.
\end{proof}

We have thus proved:

\begin{theorem}
\label{effective-iml}
Let $G$ be a commutative algebraic group defined over a finite field~$\mathbb F_q$, 
let $F:G\to G$ be the $q$-power Frobenius, let $X\subseteq G$ be a closed subvariety defined over a field extension of $\mathbb F_q$, and let $\Gamma\leq G$ be a finitely generated $\bZ[F]$-submodule. 
Suppose $(r,\Sigma,\mathcal A)$ are such that
\begin{itemize}
\item
 $q^r$ is an upper bound on the total degree of a given defining set of polynomials over $\Fq$ for the group multiplication on $G$,
\item
$\Sigma$ is a weak $F^r$-spanning set for $\Gamma$,
\item
and $\mathcal A$ is the automaton described above.
\end{itemize}
Then $X\cap\Gamma=[\mathcal L]_{F^r}$ where $\mathcal L\subseteq \Sigma^*$ is the language recognised by $\mathcal A$.
\end{theorem}

\begin{proof}
By Lemma~\ref{l=la},
$\mathcal L=\{w\in \Sigma^*:[w]_{F^r}\in X\cap\Gamma\}$.
Since $\Sigma$ is a weak $F^r$-spanning set, $[\Sigma^*]_{F^r}=\Gamma$, and so $X\cap\Gamma=[\mathcal L]_{F^r}$, as desired.
\end{proof}

As a consequence we can deduce the following generalisation of Corollary~\ref{iml-auto}.

\begin{corollary}
\label{iml-fauto-module}
Suppose $G$ is a commutative algebraic group over~$\mathbb F_q$,
$F:G\to G$ is the $q$-power Frobenius,
$X\subseteq G$ is a closed subvariety defined over a field extension of $\mathbb F_q$,
and $\Gamma\leq G$ is a finitely generated $\bZ[F]$-submodule. 
Then $X\cap \Gamma$ is $F$-automatic.
\end{corollary}

\begin{proof}
By Theorem~\ref{spanning}, we have an $r>0$ and an $F^r$-spanning set $\Sigma$ for $\Gamma$.
By~\cite[Lemma~5.7]{fsets-SML}, for every $m>0$, $\Sigma(m,F^r)$ is an $F^{rm}$-spanning set.
We may therefore assume that $r$ is sufficiently large so that there are defining equations for the multiplication on $G$ of degree less than $q^r$.
Let $\mathcal A$ be the automaton constructed above with this $\Sigma$.
Now apply Theorem~\ref{effective-iml} to $(r,\Sigma,\mathcal A)$ so that $X\cap\Gamma=[\mathcal L]_{F^r}$.
But Proposition~6.8(b) of~\cite{fsets-SML} tells us that the expansion of a regular language on an alphabet that is a spanning set is $F$-automatic.
Hence, as $\Sigma$ is an actual $F^r$-spanning set, and not just a weak one,  $[\mathcal L]_{F^r}$ is $F$-automatic.
\end{proof}

\medskip
\subsection{Effectivity}
\label{subsect-effectiveML}
Corollary~\ref{iml-fauto-module} is more general than Corollary~\ref{iml-auto} in that $\Gamma$ is only assumed to be a finitely generated $\zf$-submodule rather than an ($F$-invariant) finitely generated subgroup.
But the real gain here is that we obtain an effective description of $X\cap \Gamma$ in Theorem~\ref{effective-iml}.
Let us reiterate the grounds for this effectivity claim.
We are given the following data:
\begin{itemize}
\item
defining polynomials for the algebraic group $G$ and the subvariety $X$, embedded as locally Zariski closed subsets of $\mathbb P^n$
\item
generators for the $\zf$-submodue $\Gamma$, and
\item
a finite presentation of~$K$ over $\Fq$.
\end{itemize}
From $n$ and the generators of $\Gamma$, Proposition~\ref{explcitiwspan-iml} gives us an effectively bounded $r>0$ and an explicit defining expression for a weak $F^r$-spanning set $\Sigma$.
Then, from $(r,\Sigma)$ together with the given presentations of $G,X,$ and $K$, Proposition~\ref{bound} gives us an effective bound on the size of the set $\mathcal T_K$ which are the states of our automaton $\mathcal A$.
The automaton itself is then explicitly constructed.
Finally Theorem~\ref{effective-iml} describes $X\cap\Gamma$ as the set of $F^r$-expansions of the words on $\Sigma$ accepted by $\mathcal A$.

\bigskip
\section{Deciding rational points on subvarieties of
isotrivial abelian varieties}
\label{sect-decide}
\noindent
To illustrate the usefulness of the effective description of $X\cap\Gamma$ given by Theorem~\ref{effective-iml}, we now solve some natural decision problems in the arithmetic geometry of abelian varieties over finite fields.
Fix an abelian variety $G$ over a finite field $\Fq$, a function field extension $K$ of $\Fq$, and a closed subvariety $X\subseteq G$ defined over~$K$.
Understanding the set of rational points $X(K)$ is a fundamental problem in diophantine geometry.
Given presentations of $G, X$, and $K$, we will give decision procedures for the following three questions:
Is $X(K)$ empty? Is it infinite? Does it contain a coset of an infinite subgroup of~$G$?

Note that Theorem~\ref{effective-iml} does apply to this context as $X(K)=X\cap\Gamma$ where $\Gamma:=G(K)$ is (by Lang-N\'eron) a finitely generated $F$-invariant subgroup of $G$.
(As usual $F:G\to G$ denotes the $q$-power Frobenius.)

\medskip
\subsection{Finding generators for $G(K)$}
\label{subsect-gkgen}
The effectivity of our description of $X\cap\Gamma$ required as input also generators for $\Gamma$.
In this case, when $\Gamma$ is the set of rational points on an isotrivial abelian variety, there is already an algorithm for computing generators of $G(K)$.
Indeed, such an algorithm exists whenever the Tate-Shafarevich group is finite (see, for example,~\cite[Chapter~X]{JS}, for the case of number fields).
That the the Tate-Shafarevich group of an abelian variety over a finite field is finite is a theorem of Milne~\cite{milne68}.

We suppose, therefore, that we have computed generators $\gamma_1,\dots, \gamma_n$ for $G(K)$.

\medskip
\subsection{Whether or not $X(K)$ is empty is decidable.}
From the generators for $G(K)$, Proposition~\ref{explcitiwspan-iml} gives us an explicit $r$ and $\Sigma$ such that $\Sigma$ is a weak $F^r$-spanning set for $G(K)$.
Let $\mathcal A$ be the automaton built in~$\S$\ref{sect-automaton} from $(r,\Sigma)$.
Theorem~\ref{effective-iml} tells us that $X(K)=[\mathcal L]_{F^r}$ where $\mathcal L\subseteq \Sigma^*$ is the language recognised by $\mathcal A$.
So $X(K)$ is nonempty if and only if there is in $\mathcal A$ a path from the initial state to an accepting state.
This is a decidable property of the shape of the automaton $\mathcal A$.
\qed

\medskip
\subsection{Whether or not $X(K)$ is infinite is decidable.}
\label{decide-infinite}
This requires a certain refinement of the automaton built in~$\S$\ref{sect-automaton}.
We make use of the following lemma that allows us to reduce any given regular language modulo an $F$-automatic equivalence relation.
In fact we only need the lemma for the trivial equivalence relation of equality right now, but in~$\S$\ref{sect-iml=mod} below we will apply it to other equivalence relations, and so we do it in generality.

\begin{lemma}
\label{unamb} 
Suppose $F$ is an injective endomorphism of an abelian group $M$, and $\Sigma\subseteq M$ is a finite set.
Suppose $\sim$ is an equivalence relation on $M$ such that $\mathcal G:=\{(v,w)\in(\Sigma\times \Sigma)^* : [v]_F\sim [w]_F\}$ is regular.
Then there is a regular language $\mathcal L_0\subseteq\Sigma^*$ such that for every $w\in \Sigma^*$ there is a unique $v\in \mathcal L_0$ with $|v|\leq |w|$ and such that $[v]_F\sim [w]_F$.
\label{unambiguous}
\end{lemma}

\begin{proof}
Note that we are identifying $(\Sigma\times\Sigma)^*$ with the $\{(v,w)\in\Sigma^*\times\Sigma^*:|v|=|w|\}$ in the natural way.

Fix a total ordering on $\Sigma$ in which $0$ is least, and let $\prec$ denote the induced total ordering on $\Sigma^*$ that first orders by length and then within a given length orders lexicographically reading right to left.
It is not hard to construct the automaton witnessing that
$\mathcal F:=\{(v,w)\in (\Sigma\times \Sigma)^* : v\prec w\}$
 is regular.

Let $\mathcal{E}$ be the image of $\mathcal G\cap\mathcal F$ under projection onto the second coordinate.
Observe that $\mathcal{E}$ is the set of words $w$ for which there is some $v\prec w$ of the same length such that $[v]_F\sim [w]_F$.
Let $\mathcal L_0$ be the regular language made up of words that are not in $\mathcal{E}$ and that do not end in $0$.
We show $\mathcal L_0$ works.

Suppose $w\in\Sigma^*$.
Let $u\in\Sigma^*$ be shortest such that $[u]_F\sim [w]_F$.
If $u\notin\mathcal L_0$ then, as $u$ does not end in a $0$, it must be that $u\in\mathcal E$.
Hence there exists $v\prec u$ of the same length with $[v]_F\sim [u]_F$.
Letting $v$ be $\prec$-least such we have that $v\notin\mathcal E$.
On the other hand, $v$ cannot end in $0$ as if $v=v'0$ then $v'$ would contradict the fact that $u$ was chosen of minimal length.
Hence $v\in \mathcal L_0$, as desired.

For uniqueness, suppose, toward a contradiction, that we have distinct $v_1,v_2\in\mathcal L_0$ with $[v_1]_F\sim[v_2]_F$.
We may assume that $|v_1|\leq|v_2|$.
If $|v_1|=|v_2|$ then we may assume that $v_1\prec v_2$.
Let $n\geq 0$ be such that $|v_10^n|=|v_2|$.
As $v_2$ does not end in~$0$, we must have that $v_10^n\prec v_2$.
So $v_2\in\mathcal E$, contradicting $v_2\in\mathcal L_0$.
\end{proof}

\begin{remark}
If the equivalence relation $\sim$ is just equality then~\cite[Lemma~6.7(a)]{fsets-SML} gives that $\mathcal G$ is regular -- and indeed it explicitly constructs an automaton recognising~$\mathcal G$ -- and hence the above lemma applied to a weak $F$-spanning set $\Sigma$ gives us a regular language in which every element of $M$ has a unique base $F$ representation.
\end{remark}

We are now ready to decide the problem of whether $X(K)$ is infinite.
As before, we first produce $(r,\Sigma,\mathcal A)$ such that $\Sigma$ is a weak $F^r$-spanning set for $G(K)$ and  $\mathcal A$ is the automaton built in~$\S$\ref{sect-automaton} from $(r,\Sigma)$.
Let  $\mathcal L\subseteq \Sigma^*$ be the language recognised by~$\mathcal A$.
So $X(K)=[\mathcal L]_{F^r}$ by Theorem~\ref{effective-iml}.
Let $\mathcal L_0\subseteq \Sigma^*$ be the regular language given by Lemma~\ref{unamb} applied to $(G(K), F^r)$ with the trivial equivalence relation (namely, equality).
So $\mathcal L_0$ is regular, and indeed, by following the proof of Lemma~\ref{unamb}, one can explicitly construct an automaton $\mathcal A_0$ that recognises it.
Let $\mathcal L':=\mathcal L\cap\mathcal L_0$ and let $\mathcal A'$ be the automaton obtained from $\mathcal A$ and $\mathcal A_0$ that recognises $\mathcal L'$.
By construction, $w\mapsto[w]_{F^r}$ is a bijection between $\mathcal L'$ and $X(K)$.
So $X(K)$ is infinite if and only if $\mathcal L'$ is.
By the pumping lemma, $\mathcal L'$ is infinite if and only if there exist strings $u,v,w\in \Sigma^*$ with $|v|\ge 1$ such that $uv^i w\in \mathcal L'$ for all $i\ge 0$.
See, for example, \cite[Lemma 4.2.1]{shallit}.
This is something decidable about the shape of the automaton recognising $\mathcal L'$; namely $\mathcal L'$ is infinite if and only if there is a path from the initial state of $\mathcal A'$ to an accepting state that includes a nontrivial loop.
\qed

\medskip
\subsection{Finding an actual spanning set for $G(K)$}
\label{subsect-effectivespan}
In order to deal with our last decision problem -- whether or not $X(K)$ contains an infinite coset -- weak spanning sets will not be sufficient.
This is because we will be making use of some recent results of Christopher Hawthorne~\cite{hawthorne} that require actual spanning sets.
Fortunately, using methods from~\cite{fsets-SML}, we can give an effective algorithm for producing a spanning set for $G(K)$.
(Our method does not extend to arbitrary $\zf$-submodules of $G(K)$, and that is why we only obtained weak spanning sets in~$\S$\ref{sect-ewss}.)
We give some details.

In~$\S$\ref{subsect-effectivity} we showed how to effectively construct an $\Fq$-algebra $R$ with an $\Fq$-subspace $W$ such that $R=\Fq[W]$ and $K=\operatorname{Frac}(R)$.
It was also explained how we can then work effectively with the corresponding height function $\height_W$ given by Definition~\ref{ht0}.
For each positive integer $N$, set
$$\Sigma_N:=\{x\in G(K):\height_W(x)\leq N\text{ or }\height_W(-x)\leq N\}.$$
It is shown in~\cite[Proposition~5.8]{fsets-SML} that this set will be an $F^r$-spanning set for some~$r$ and sufficiently large $N$.
We need to produce such $r$ and $N$ effectively.

In fact, the proof of~\cite[Proposition~5.8]{fsets-SML} tells us what $r$ should be: the requirement is that $6C^6\leq D^r$ where  $C,D>1$ are integers satisfying:
\begin{eqnarray}
\label{d1}
\height_W(x+y)&\leq&C(\height_W(x)+\height_W(y))\text{ for all }x,y\in G(K),\\
\label{d2}
\height_W(-x)&\leq& C\height_W(x)\text{ for all }x\in G(K),\\
\label{cz}
\height_W(F(x))&\geq& D\height_W(x)\text{ for all but finitely many }x\in G(K).
\end{eqnarray}

\begin{remark}
Note that in our setting the $\kappa$ of~\cite[Proposition~5.8]{fsets-SML} is $0$ and hence the conditions~(i)--(iii) in the proof there are consequences of $6C^6\leq D^r$.
\end{remark}

So to find $r$ we need to effectively find such $C$ and $D$.
The proof of~\cite[Corollary~5.9]{fsets-SML} shows that~(\ref{d1}) and~(\ref{d2}) hold with $C$ being the maximum of the degrees of the given polynomials defining the group law and inverse on $G$.
We now describe how to find~$D$ witnessing~(\ref{cz}).

First, recall that, for each positive integer $\ell$, we use $W^\ell$ to denote be the $\Fq$-vector space spanned by the $\ell$-fold products of elements from $W$, while we use $W^{\langle q\rangle}$ to denote the $\Fq$-vector space of $q$-powers of elements from $W$.

By construction (see Lemma~\ref{R}), we have a fixed $K^q$-basis $h_1,\dots,h_m$ for $K$ such that $R=\Fq[W]=\sum_{j=1}^m\Fq[W]^qh_j$.
Letting $\lambda_1,\dots,\lambda_m$ denote the operators corresponding to $h_1,\dots,h_m$,
let
$$t_0:=\max\{\height_W(\lambda_j(v)):v\in W^{\dim W(q-1)}, j=1,\dots,m\}.$$
It follows that
$$W^{\dim W(q-1)}\subseteq\sum_{j=1}^m(W^{t_0})^{\langle q\rangle}h_j.$$
The proof of Lemma~\ref{lambdaW} then shows that
$$W^t\subseteq\sum_{j=1}^m(W^{\lfloor\frac{t}{q}\rfloor+t_0})^{\langle q\rangle}h_j$$
for all $t\geq 0$.
Letting $t_1:=\lceil \frac{t_0q+q}{q-1}\rceil$ one computes that $\frac{t_1\delta}{q}+t_0\leq(t_1-1)\delta$ for all $\delta>0$, and
hence
$$W^{t_1\delta}\subseteq\sum_{j=1}^m(W^{(t_1-1)\delta})^{\langle q\rangle}h_j$$
for all $\delta>0$.
Letting $D:=\frac{t_1^2}{t_1^2-1}$ and $Z_0:=\{z\in \Fq[W]:\height_W(z)\leq t_1^2+t_1\}$, the proof of~\cite[Corollary~5.9]{fsets-SML} then shows that
$\height_W(x^q)\geq D\height_W(x)$ for all $x\in \Fq[W]\setminus Z_0$.
Hence $D$ witnesses~(\ref{cz}), with the finite exceptional set $Z$ being those elements of $G(K)$ that have a homogeneous representation of the form $[x_0:\cdots:x_n]$ with each $x_i\in Z_0$.

Having found an $r$ that works, the proof of \cite[Proposition~5.8]{fsets-SML} also tells us how large $N$ has to be: it must be big enough so that $\Sigma_N$ contains an exceptional set for~(\ref{cz}) as well as a complete set of representatives for $G\big(K^{q^r}\big)$ in $G(K)$.
We have already effectively produced an exceptional set for~(\ref{cz}), namely $Z$.
For representatives of $G\big(K^{q^r}\big)$ in $G(K)$ we make use of the set of generators $\gamma_1,\dots,\gamma_n$ for $G(K)$ computed in~$\S$\ref{subsect-gkgen}.
From the fact that multiplication by $q^r$ on $G$ is of the form $F^r\circ V^r$ where $V$ is the Verschiebung, we get that $[q^r]G(K)\leq G(K^{q^r})$.
Hence $\{[\ell]\gamma_i:1\leq i\leq n, 0\leq \ell\leq q^r\}$ is a complete set of representatives for $G\big(K^{q^r}\big)$ in~$G(K)$.
Taking $N$ to be greater than the height of all these elements, as well as all the elements of $Z$, we have that $\Sigma_N$ is an $F^r$-spanning set for $G(K)$.

\medskip
\subsection{Whether or not $X(K)$ contains an infinite coset is decidable.}
\label{decide-coset}
The strategy for deciding this will be as follows:
We have just computed $(r,\Sigma)$ such that $\Sigma$ is an $F^r$-spanning set for $G(K)$.
As in~$\S$\ref{decide-infinite}, we can then produce a regular $\mathcal L'\subseteq\Sigma^*$ recognised by an explicitly constructed automaton $\mathcal A'$, such that $w\mapsto[w]_{F^r}$ is a bijection between $\mathcal L'$ and $X(K)$.
We saw there that the infinitude of $\mathcal L'$ was reflected in the shape of $\mathcal A'$.
Similarly, we will show that $X(K)$ contains an infinite coset if and only if $\mathcal L'$ is not ``sparse", and that sparsity can be read off from~$\mathcal A'$.

First, we need to recall what it means for a language to be sparse.
A sublanguage $\mathcal L\subseteq\Sigma^*$ is said to be {\em sparse} if it is regular and the number of words in $\mathcal L$ of length $n$ is $O(n^d)$ as $n$ grows, for some $d\geq 0$.
A list of equivalent formulations that we will make use of is compiled in~\cite[Proposition~7.1]{fsets-SML}.
Sparsity in the general setting of $F$-automatic sets has been studied further recently by Hawthorne in~\cite{hawthorne}, from which we will use the following two results:

\begin{fact}[Corollary~5.12 of~\cite{hawthorne}]
\label{robust-sparse}
Suppose $(M,F)$ is an abelian group equipped with an injective endomorphism and $S\subseteq M$ is $F$-automatic.
The following are equivalent conditions on~$S$:
\begin{itemize}
\item[(i)]
There exists an $r>0$, an $F^r$-spanning set $\Sigma$  for $M$, and a sparse sublanguage $\mathcal L\subseteq\Sigma^*$, such that $S=[\mathcal L]_{F^r}$.
We say that $S$ is {\em $F$-sparse}.
\item[(ii)]
Suppose $r>0$, $\Sigma$ is an $F^r$-spanning set for $M$, and $\mathcal L\subseteq \Sigma^*$ is regular such that $[\mathcal L]_{F^r}=S$.
Fix any linear ordering $\prec$ on $\Sigma$ and denote also by $\prec$ the induced length-lexicographic ordering on $\Sigma^*$.
Then 
$$\displaystyle \mathcal L_\prec:=\{w\in\mathcal L:w\preceq v\text{ for all }v\in\mathcal L\text{ with }[v]_{F^r}=[w]_{F^r}\}$$
is sparse.
\end{itemize}
\end{fact}

\begin{remark}
By the {\em length-lexicographic} ordering $\prec$ we mean that among words of a given length the ordering is lexicographic and if $|w|<|v|$ then $w\prec v$.
\end{remark}

Subsets satisfying condition~(i) were called {\em $F$-sparse} in~\cite{fsets-SML}, but the seemingly stronger condition~(ii) is more concrete and easier both to verify or falsify in practice.
In $\S$\ref{subsect-sparseorb} we will return to sparsity, studying in detail the structure of $F$-sparse subsets of isotrivial commutative algebraic groups. 
The next fact says that $F$-sparse sets are far from being subgroups.

\begin{fact}[Corollary~5.14 of~\cite{hawthorne}]
\label{nosparsegroup}
Suppose $(M,F)$ is a finitely generated abelian group equipped with an injective endomorphism and $S\subseteq M$ is $F$-automatic.
If $S$ is $F$-sparse then it does not contain any coset of an infinite subgroup of~$M$.
\end{fact}

We now use $F$-sparsity to solve our decision problem.
Recall that we already effectively produced an $F^r$-spanning set $\Sigma$ for $G(K)$, and a regular language $\mathcal L'$ on $\Sigma$ such that $X(K)=[\mathcal L']_{F^r}$ and each element of $X(K)$ has a unique $F^r$-representation in $\mathcal L'$.
We first argue that $X(K)$ contains an infinite coset if and only if $\mathcal L'$ is not sparse.
Indeed, assume that $X(K)$ contains a coset of an infinite subgroup of $G$.
Then in fact that subgroup is in $G(K)$, and hence by Fact~\ref{nosparsegroup} applied to $M=G(K)$ and $S=X(K)$, it follows that $X(K)$ is not $F$-sparse.
Hence $\mathcal L'$ is not sparse.
Conversely, suppose that $X(K)$ does not contain any such infinite coset.
Then by Theorem~\ref{iml} -- or indeed by the original isotrivial Mordell-Lang theorem~\cite[Theorem~B]{fsets} -- we have that $X(K)$ is a finite union of sets from $\mathcal S(G,F)$.
But it was shown in~\cite{fsets-SML}, see the proof of Theorem~7.4 of that paper, that such sets are $F$-sparse.
So condition~(ii) of Fact~\ref{robust-sparse} holds of $S=X(K)$, and we get that $\mathcal L'_\prec$ is also sparse.
But $\mathcal L'_\prec=\mathcal L'$ as each element of $X(K)$ has a unique $F^r$-representation in $\mathcal L'$.
Therefore $\mathcal L'$ is sparse.

We have reduced the problem to deciding whether or not $\mathcal L'$ is sparse.
Let $\mathcal A'$ be the automaton recognising $\mathcal L'$ that was explicitly constructed in the proof of Theorem~\ref{decide-infinite}.
We may assume that every state in $\mathcal A'$ is accessible, in the sense that it can be reached by some input.
Now, one of the equivalent conditions for $\mathcal L'$ to be sparse, given in~\cite[Proposition~7.1(4)]{fsets-SML}, is that $\mathcal A'$ has no ``double loops".
This condition is something visibly decidable about the shape of $\mathcal A'$.
\qed

\bigskip
\section{A gap theorem for rational points of bounded height on subvarieties of  isotrivial abelian varieties}
\label{height dichotomy}

\noindent
The sparse/non-sparse dichotomy for regular languages that appeared in~$\S$\ref{decide-coset} yields an analogous dichotomy within the context of isotrivial Mordell-Lang.
In order to properly state this, we make use of the N\'eron-Tate canonical height on an abelian variety (see \cite{Lang} or \cite[Part~B]{H-S} for more details).

\begin{theorem}
Let $G$ be an abelian variety defined over $\mathbb{F}_q$, let $F:G\to G$ be the $q$-power Frobenius endomorphism, let $K$ be a finitely generated field extension of~$\mathbb{F}_q$, and let $X$ be a closed subvariety of $G$ over $K$.
Denote by $h$ the N\'eron-Tate canonical height on $G$.
Then the following are equivalent:
\begin{enumerate}
\item $X(K)$ is $F$-sparse;
\item there are $C>0$ and $d\ge 0$ such that for sufficiently large $H$
$$\#\{ c\in X(K)\colon h(c)\le H\} \le C(\log(H))^d;$$
\item 
$\#\{ c\in X(K)\colon h(c)\le H\} ={\rm o}(H^{1/2})$;
\item $X(K)$ does not contain a coset of an infinite subgroup of $G$;
\item $X(K)$ is a finite union of sets from $\mathcal S(G,F)$.
\end{enumerate}
\label{thm:equivalence}
\end{theorem}

\begin{remark}
The equivalence of~(2) and~(3) means that there is a gap: either the number of points on $X(K)$ of height at most $H$ is bounded above by $C(\log(H))^d$ or it is bounded below by $C' H^{1/2}$.
\end{remark}

To prove this result, we require two technical lemmas dealing with estimation.
We have a positive quadratic form $\langle\, \cdot\, , \, \cdot\, \rangle$ on $G(K)\times G(K)$ given by
$$\langle x,y\rangle = h(x+y)-h(x)-h(y).$$
Suppose $\Sigma\subseteq G(K)$ is a finite set.
Observe that for $c\in [\Sigma]_F$ and $x\in G(K)$,
$$2h(x+c)=\langle x+c,x+c\rangle =\langle x,x\rangle + 2\langle x,c\rangle +\langle c,c\rangle.$$
We can rewrite the right side as $2h(x)+2h(c)+2\langle x,c\rangle$.
Then since we have only finitely many choices for $c\in [\Sigma]_F$, Cauchy-Schwarz gives there is some positive constant $\kappa_1$ such that 
$\langle x,c\rangle \le \kappa_1 h(x)^{1/2}$.
So letting $\kappa_2$ denote the max of $h(c)$ as $c$ ranges over $[\Sigma]_F$ we get 
 there are positive constants $\kappa_1,\kappa_2$ such that,
\begin{equation}
\label{eq:kappa12}
h(x+c)\le h(x)+\kappa_1 h(x)^{1/2} + \kappa_2\ \ \ \text{ for all $c\in [\Sigma]_F$ and $x\in G(K)$.}
\end{equation}
In addition (using that the N\'eron-Tate height differs from the usual Weil height by a uniform bounded amount), if $r$ is a positive integer, then there is a constant $\kappa_3$, which depends upon $r$, such that 
\begin{equation}
\label{eq:kappa3}
q^r h(x)-\kappa_3\le h(F^r(x))\le q^r h(x)+\kappa_3\ \ \ \text{ for all $x\in G(K)$.}
\end{equation}

\begin{lemma}
Let $\Sigma$ be an $F^r$-spanning set for $G(K)$ and $\kappa_1,\kappa_2,\kappa_3$ as in Equations~(\ref{eq:kappa12}) and~(\ref{eq:kappa3}).
Define $E_n$ recursively by $E_0=\max\{2+\kappa_3+h(a)\colon a\in \Sigma\}$ and $E_n  = E_{n-1}(1+(\kappa_1+\kappa_3)/q^{n/2})$ for $n\ge 1$.
Then for all $c_0,\ldots ,c_n\in \Sigma$, the height of 
$c_0 + F^r(c_1)+\cdots +F^{nr}(c_n)$ is at most
$E_n q^{nr}$.
In particular, there is a positive constant $C_0$ such that for every $n\ge 1$ we have
$c_0 + F^r(c_1)+\cdots +F^{nr}(c_n)$ has height at most $C_0q^{nr}$.
\label{lem:k12}
\end{lemma}

\begin{proof}
We prove the first part by induction on $n$.  When $n=0$ this follows from our choice of $E_0$.
Now suppose that the result holds whenever $n<k$ for some $k\ge 1$ and consider an element $c_0+\cdots +F^{kr}(c_k)$.
Then
$$F^r(c_1)+\cdots +F^{kr}(c_k) = F^r(c_1+\cdots +F^{r(k-1)}(c_k))$$
has height at most
$q^{kr} E_{k-1}+ \kappa_3$ by the induction hypothesis and Equation (\ref{eq:kappa3}).  
We then translate by $c_0\in \Sigma$ and so an application of Equation (\ref{eq:kappa12}) yields that the height of 
$c_0+\cdots +F^{kr}(c_k)$ is at most
$q^{kr} E_{k-1}+\kappa_3 + \kappa_1(q^{kr} E_{k-1}+\kappa_3)^{1/2} + \kappa_2$.  
This quantity is less than or equal to
$$q^{kr} E_{k-1} + (\kappa_1+\kappa_2) E_{k-1} q^{kr/2} = q^{kr} E_k,$$ so we obtain the first claim.
To complete the proof observe that
$E_k\le C_0:=E_0\prod_{j\ge 1} (1+(\kappa_1+\kappa_2)/q^{jr/2})$ and the product on the right converges since
$\sum 1/q^{jr/2}$ converges, so we have that the height of $F^r$-expansions of length $n$ is at most 
$C_0q^{nr}$ for every $n\ge 0$.  The result now follows.
\end{proof}

\begin{lemma}
Let $\Sigma$ be an $F^r$-spanning set for $G(K)$ and $\kappa_1,\kappa_2,\kappa_3$ as in Equations~(\ref{eq:kappa12}) and~(\ref{eq:kappa3}). Then there is a positive constant $C_1$ such that every element of $G(K)$ of height at most $C_1q^{nr}$ can be realised as the $F^r$-expansion of a word in $\Sigma^*$ of length at most $n$. 
\label{lem:k3}
\end{lemma}
\begin{proof}
Let $C=\kappa_1+\kappa_2+\kappa_3$.  We let $\ell$ be a positive integer with the property $q^{\ell r/2}>C$.
As before, 
$$\prod_{j\ge \ell}(1-C/q^{jr/2})$$ converges to a positive number $\theta>0$.  Pick $B_{\ell}$ such that $q^{-\ell r} B_{\ell} \theta>1$. Then there is some $m$ such that every element of N\'eron-Tate height at most $q^{\ell r} B_{\ell}$ has an $F$-expansion of length at most $m$.  We define $B_n=B_{n-1}(1-C/q^{nr/2})$ for $n>\ell$.  So from the above we have $B_n>1$ for all $n\ge \ell$.
 
 We claim that for every $n\ge \ell$ all numbers of height at most $q^{nr} B_n$ can be realised as the $F^r$-expansion of a word in $\Sigma^*$ of length at most $m+n$.  We again prove this by induction on $n$ with the base case, $n=\ell$, being immediate.  Now suppose that the claim holds whenever $n<k$ and suppose that $x$ has height at most $q^{kr} B_k$.  Then there is some $c_0\in \Sigma$ such that $x-c_0\in G(K^{q^r})$.  Hence 
$h(x-c_0)\le h(x) + \kappa_1 h(x)^{1/2} + \kappa_2$.  So if we let $y=F^{-r}(x-c_0)$ then 
$$h(y) \le \frac{h(x) + \kappa_1 h(x)^{1/2} + \kappa_2 +\kappa_3}{q}.$$  Since $q^{kr} B_k\ge 1$,
the height of $y$ is at most
$(q^{kr}B_k + (\kappa_1+\kappa_2+\kappa_3) (q^{kr} B_k)^{1/2})/q^r$.
Now we claim that $B_{k-1}\ge (B_k + (\kappa_1+\kappa_2+\kappa_3) (B_k/q^{kr})^{1/2})$.  Once we have this, we will have that $y$ has height at most $q^{(k-1)r}B_{k-1}$ and so the result will then follow by induction.

To show this we must show $B_{k-1}\ge B_k+ C (B_k/q^{kr})^{1/2}= B_k(1+C (B_kq^{kr})^{-1/2})$. 
But $B_k = B_{k-1} (1-C/q^{kr/2})$ and so 
\begin{align*}
B_k(1+C (B_k q^{kr})^{-1/2}) &= B_{k-1}(1-C/q^{kr/2}) + C B_k^{1/2} /q^{kr/2}\\
&\le B_{k-1} - C B_{k-1}/q^{kr/2} + C B_{k-1}^{1/2}/q^{kr/2} \\
&\le B_{k-1}.
\end{align*}
So since $B_k\ge 1$ for all $k$ we see that all elements of height at most $q^{kr}$ have an expansion of length at most $m+k$ and so we see that if we take $C_1=q^{-mr}$ we obtain the result.
\end{proof}

\begin{corollary}
\label{gap-corollary}
Let $\Sigma$ be an $F^r$-spanning set for $G(K)$ and let
$$\mathcal L:=\{v\in \Sigma^*\colon [v]_{F^r}\in X(K)\}.$$
Fix a linear ordering $\prec$ on $\Sigma$ and denote also by $\prec$ the induced length-lexicographic ordering on $\Sigma^*$.
Let
$$\displaystyle \mathcal E:=\{w\in\mathcal L:w\preceq v\text{ for all }v\in\mathcal L\text{ with }[v]_{F^r}=[w]_{F^r}\}.$$
If $\Theta(m)$ denotes the number of elements in $
X(K)$ of height at most $m$, then there are positive constants $C_0$ and $C_1$ such that
\begin{equation}\label{eq:Theta}
\Theta(C_1 q^{nr})  \le \#\mathcal{E}_{\le n} \le \Theta(C_0 q^{nr})\end{equation}
for all $n\ge 1$, where $\mathcal{E}_{\le n}$ is the set of words of length at most $n$ in $\mathcal{E}$.
\end{corollary}

\begin{proof}
By Lemma \ref{lem:k3}, there is a positive constant $C_1$ such that every element of $G(K)$ of height at most $C_1q^{nr}$ is of the form $[v]_{F^r}$ with $v\in \Sigma^*$ of length at most~$n$.
In particular, for each $x\in X(K)$ with $h(x)\le C_1 q^{nr}$ there is a word $w\in \mathcal{E}$ of length at most $n$ such that $[w]_F=c$.
This gives the first inequality.
By Lemma \ref{lem:k12}, there is a positive constant $C_0$ such that 
every word $w$ in $\Sigma^*$ of length at most $n$ has the property that $h([w]_{F^r})\le C_0 q^{nr}$.
Since $w\mapsto [w]_{F^r}$ is a bijection between $\mathcal E$ and $X(K)$, this gives the second inequality.
\end{proof}

We can now prove the gap theorem.

\begin{proof}[Proof of Theorem \ref{thm:equivalence}]
By Corollary~\ref{iml-auto}, or indeed by~\cite[Corollary~6.10]{fsets-SML}, we have that there exists an $r>0$ and an $F^r$-spanning set $\Sigma$ for $G(K)$ such that
$$\mathcal L:=\{v\in \Sigma^*\colon [v]_{F^r}\in X(K)\}$$
is regular.
So Corollary~\ref{gap-corollary} applies.

Suppose that $(1)$ holds and let $\mathcal E\subseteq\mathcal L$ be as in Corollary~\ref{gap-corollary}.
By Fact~\ref{robust-sparse}, noting that $\mathcal E$ is just what was called $\mathcal L_{\prec}$ there, we have that $\mathcal E$ is sparse.
Hence $\#\mathcal{E}_{\le n}$ is polynomially bounded, and so by Equation~(\ref{eq:Theta})
we see that $(2)$ holds.

The implication $(2)\implies (3)$ is immediate.

Observe that if $(4)$ does not hold then $X(K)$ contains $\{a+nb\colon n\in \mathbb{Z}\}$ for some $a,b\in G(K)$ with $b$ non-torsion and since $a+nb$ has height ${\rm O}(n^2)$ we see that $(3)$ does not hold.
So $(3)$ implies $(4)$.

That $(4)\implies(5)$ follows from Theorem~\ref{iml}, or indeed from~\cite[Theorem~B]{fsets}.

Finally, that $(5)\implies (1)$ was already pointed out in~\S\ref{decide-coset} using results of~\cite{fsets-SML}.
\end{proof}

\begin{remark}
Theorem \ref{thm:equivalence} holds more generally for algebraic groups $G$ that have a height function for which Equations (\ref{eq:kappa12}) and (\ref{eq:kappa3}) hold after one replaces $G(K)$ with a finitely generated $F$-submodule $\Gamma$ of $G(K)$.  In particular, the equivalences hold for $G(R)$ when $G=\mathbb{G}_m^d$ and $R$ is a finitely generated $\mathbb{F}_q$-algebra that is an integral domain.
\end{remark}
\bigskip
\section{Mordell-Lang for finitely generated $\zf$-submodules of isotrivial commutative algebraic groups}
\label{sect-iml=mod}
\noindent
Recall that $\mathcal S(G,F)$ denotes the collection of translates of finite sums of $F$-orbits, i.e., subsets of the form $a+S(a_1,\dots,a_r;\delta_1,\dots,\delta_r)$ in the notation of Section~\ref{sect-ml}.
Our goal in this final section is to prove the following:

\begin{theorem}
\label{iml-module}
Let $G$ be a commutative algebraic group over a finite field $\mathbb F_q$, let $F:G\to G$ be the $q$-power Frobenius, let $X\subseteq G$ be a closed subvariety defined over a field extension of $\mathbb F_q$, and let  $\Gamma\leq G$ be a finitely generated $\zf$-submodule.
Then $X\cap\Gamma$ is a finite union of sets of the form $S+\Lambda$ where $S\subseteq\Gamma$ is in $\mathcal S(G,F)$ and $\Lambda\leq\Gamma$ is a $\mathbb Z[F^r]$-submodule for some $r>0$.
\end{theorem}

\begin{remark}
\label{betterconclusion}
In the conclusion of the theorem we can replace $\Lambda$ by $H\cap\Gamma$ where $H$ is the Zariski closure of $\Lambda$.
Indeed, $S+\Lambda\subseteq S+(H\cap\Gamma)\subseteq X\cap\Gamma$.
Note also that $H\leq G$ is an algebraic subgroup over a finite field since $\Lambda$ is $F^r$-invariant.
\end{remark}

This is an isotrivial Mordell-Lang statement like Theorem~\ref{iml} but for finitely generated $\zf$-submodules rather than finitely generated groups that are $F$-invariant.
When we restrict to $G$ semiabelian, Theorem~\ref{iml-module} is precisely Theorem~\ref{iml-semiabelian-intro}, the original isotrivial Mordell-Lang theorem of the third author and Thomas Scanlon from~\cite{fsets}.
However, our proof even in that case is new; we deduce the combinatorial structure of $X\cap\Gamma$ directly from the $F$-automaticity given by Corollary~\ref{iml-fauto-module}.

The proof will proceed by a series of reductions, with the key technical step being an understanding of the structure of ``$F$-sparse" subsets of $\Gamma$.

\medskip
\subsection{Sparsity and $F$-orbits}
\label{subsect-sparseorb}
Suppose we have an abelian group $M$ equipped with an injective endomorphism $F:M\to M$ and $\Sigma$ a finite subset of $M$.
Recall that $\mathcal L\subseteq\Sigma^*$ is said to be {\em sparse} if it is regular and the number of words in $\mathcal L$ of length $n$ is $O(n^d)$ as $n$ grows, for some $d\geq 0$.
A list of equivalent formulations is compiled in~\cite[Proposition~7.1]{fsets-SML}.
In particular, it is shown there that every sparse language is a finite union of languages of the form
$$u_1 w_1^* u_2 w_2^* \cdots u_m w_m^* u_{m+1}:=\{u_1 w_1^{n_1} u_2 w_2^{n_2} \cdots u_m w_m^{n_m} u_{m+1}:n_1,\dots,n_m\geq 0\}$$
for some words $u_i,w_i\in\Sigma^*$ with the $w_i$ all nontrivial.
Let us call languages of this form {\em simple sparse}.
The $F$-expansions of a simple sparse language will have a very special form.

\begin{definition}
Given $a_1,\dots,a_r\in M$ and positive integers $\delta_1,\dots,\delta_r$, we denote
$$\left\{F^{n_1\delta_1}a_1+F^{n_1\delta_1+n_2\delta_2}a_2+\cdots+F^{\sum_{i=1}^rn_i\delta_i}a_r:n_1,\dots,n_r\geq 0\right\}$$
by $E(a_1,\dots,a_r;\delta_1,\dots,\delta_r)$.
As usual, we write $E(a_1,\dots,a_r;\delta)$ in the case when all the $\delta_i=\delta$.
\end{definition}

\begin{lemma}
\label{expandsparse}
Suppose $M$ is an abelian group, $F:M\to M$ is an injective endomorphism, and $\Sigma$ is a finite subset of $M$.
Suppose $(N,F)$ is an extension of $(M,F)$ for which $F-\id:N\to N$ is surjective.

If $\mathcal L=u_1 w_1^* u_2 w_2^* \cdots u_m w_m^* u_{m+1}$ is a simple sparse sublanguage of $\Sigma^*$ then
$$[\mathcal L]_F=a_0+E(a_1,\dots,a_r;\delta_1,\dots,\delta_r)$$
for some $a_0,\dots,a_r\in N$ and $\delta_1,\dots,\delta_r>0$.
Moreover, if $m>0$, we can choose  $a_0,\dots,a_r,\delta_1,\dots,\delta_r$ above so that concatenating $\mathcal L$ with any $v\in\Sigma^*$ yields
$$[\mathcal Lv]_F=a_0+E(a_1,\dots,a_{r-1},a_r+F^\ell[v]_F;\delta_1,\dots,\delta_r)$$
where $\ell=\sum_{i=1}^{m+1}|u_i|$.
\end{lemma}

\begin{proof}
We proceed by induction on $m$.
If $m=0$ then $a_0=[u_1]_F$ and $r=0$ works.

So assume $m>0$ and write $\mathcal L$ as the concatenation $\mathcal L_1\mathcal L_2$ where $\mathcal L_1=u_1w_1^*$ and $\mathcal L_2=u_2 w_2^* \cdots u_m w_m^* u_{m+1}$.
By the induction hypothesis,
\begin{equation}
\label{l2}
[\mathcal L_2]_F=a_0+E(a_1,\dots,a_r;\delta_1,\dots,\delta_r)
\end{equation}
for some $a_0,\dots,a_r$ from $N$ and positive integers $\delta_1,\dots,\delta_r$.
Let $|u_1|=\alpha$, $|w_1|=\beta$, $x=[u_1]_F$, and $y=[w_1]_F$.
The elements of $[\mathcal L]_F$ are precisely those of the form
$[u_1w_1^n]_F+ F^{\alpha+n\beta}(z)$ with $z\in [\mathcal L_2]_F$.
Note that 
$$[u_1w_1^n]_F+ F^{\alpha+n\beta}(z)
=x+F^{\alpha}(y)+F^{\alpha+\beta}(y)+\cdots + F^{\alpha+(n-1)\beta}(y) + F^{\alpha+n\beta}(z).$$
Now, since $F-\id$ is surjective on $N$, there exists
$\theta\in N$ such that $F^{\alpha}(y) = F^{\beta}(\theta)-\theta$.
A simple telescoping argument gives $$[a_1b_1^n]_F+ F^{\alpha+n\beta}(z) = x-\theta+F^{\beta n}(\theta + F^{\alpha}(z)).$$
Setting $\gamma:=x-\theta$ we get that
\begin{eqnarray*}
[\mathcal L]_F
&=&
\bigcup_{z\in [\mathcal L_2]_F}\bigcup_{n\geq 0}\{\gamma+F^{\beta n}(\theta+F^{\alpha}(z))\}\\
&=&
\gamma+\bigcup_{n\geq 0}F^{\beta n}(\theta+F^{\alpha}([\mathcal L_2]_F))\\
&=&
\gamma+\bigcup_{n\geq 0}F^{\beta n}(\theta+F^{\alpha}a_0+E(F^{\alpha}a_1,\dots,F^{\alpha}a_r;\delta_1,\dots,\delta_r))\\
&=&
\gamma+E(\theta+F^\alpha a_0,F^\alpha a_1,\dots,F^\alpha a_r;\beta,\delta_1,\dots,\delta_r)
\end{eqnarray*}
as desired.

For the ``moreover" clause we must first consider the $m=1$.
In that case $\mathcal L_2=\{u_2\}$ and we have
\begin{eqnarray*}
[\mathcal L]_F
&=&
\bigcup_{n\geq 0}\{\gamma+F^{\beta n}(\theta+F^{\alpha}[u_2]_F)\}\\
&=&
\gamma+E(\theta+F^{\alpha}[u_2]_F;\beta)
\end{eqnarray*}
and for any $v\in\Sigma^*$, 
\begin{eqnarray*}
[\mathcal Lv]_F
&=&
\bigcup_{n\geq 0}\{\gamma+F^{\beta n}(\theta+F^{\alpha}[u_2v]_F)\}\\
&=&
\gamma+E(\theta+F^{\alpha}[u_2]_F+F^{\alpha+|u_2|}[v]_F;\beta).
\end{eqnarray*}
Since $\alpha+|u_2|=|u_1|+|u_2|$ this proves the ``moreover" clause when $m=1$.

Suppose $m>1$.
The inductive hypothesis yields that in~(\ref{l2}) we can choose $a_0,\dots,a_r,\delta_1,\dots,\delta_r$ so that for any $v\in\Sigma^*$
$$[\mathcal L_2v]_F=a_0+E(a_1,\dots,a_{r-1},a_r+F^\ell[v]_F;\delta_1,\dots,\delta_r)$$
where $\ell=\sum_{i=2}^{m+1}|u_i|$. 
Hence
\begin{eqnarray*}
[\mathcal Lv]_F
&=&
\gamma+\bigcup_{n\geq 0}F^{\beta n}(\theta+F^{\alpha}([\mathcal L_2v]_F))\\
&=&
\gamma+\bigcup_{n\geq 0}F^{\beta n}(\theta+F^{\alpha}a_0+E(F^{\alpha}a_1,\dots, F^{\alpha}(a_r+F^\ell[v]_F);\delta_1,\dots,\delta_r))\\
&=&
\gamma+E(\theta+F^\alpha a_0,F^\alpha a_1,\dots,F^\alpha a_{r-1},F^{\alpha}a_r+F^{\alpha+\ell}[v]_F;\beta,\delta_1,\dots,\delta_r).
\end{eqnarray*}
As $\alpha=|u_1|$, this proves the ``moreover'' clause.
\end{proof}

\begin{remark}
Note that while $[\mathcal L]_F\subseteq M$ we have to go to an extension to find $a_0,\dots,a_r$.
In the proof this comes from the induction step where $x$ and $y$ are in $M$ but we have to pass to $N$ to find $\gamma$ and $\theta$.
\end{remark}

The sets $E(a_1,\dots,a_r;\delta_1,\dots,\delta_r)$, despite some superficial similarity to translates of sums of $F$-orbits, are not expressible as finite unions of such.
However, in our geometric context, they can be so expressed up to Zariski closure.

\begin{proposition}
\label{et}
Suppose $G$ is a commutative algebraic group over $\Fq$, $F:G\to G$ is the $q$-power Frobenius, $K$ is a function field over $\Fq$, and $\Gamma\leq G(K)$ is an $F$-pure $\zf$-submodule.
Suppose $E:=a_0+E(a_1,\dots,a_r;\delta_1,\dots,\delta_r)\subseteq \Gamma$ where $a_0,\dots,a_r\in G$ and $\delta_1,\dots,\delta_r$ are positive integers.
Then there exists $T\subseteq \Gamma$, a finite union of sets from $\mathcal S(G,F)$, such that $E\subseteq T\subseteq\overline{E}$ where the over line denotes Zariski closure.
\end{proposition}

\begin{proof}
We first consider the case when $E=a_0+E(a_1,\dots,a_r;1)$; that is, when all the $\delta_i=1$.

Note that, as $E\subseteq\Gamma$, we have $Fa_i-a_i\in\Gamma$ for each $i=1,\dots,r$.
Indeed, for $i=r$ this is a because both $a_0+\cdots+a_{r-1}+Fa_r$ and $a_0+\cdots+a_r$ are in $E$ and hence in $\Gamma$, and one simply takes their difference.
For $i=r-1$ one notes that $a_0+\cdots+a_{r-2}+Fa_{r-1}+Fa_r\in E$ and so subtracting $a_0+\cdots+a_r$ yields that $(Fa_{r-1}-a_{r-1})+(Fa_r-a_r)\in \Gamma$, and hence by the $i=r$ case $Fa_{r-1}-a_{r-1}\in \Gamma$.
And so on all the way down to $i=1$.

It follows, by applying $F$ and taking sums, that $F^ma_i-a_i\in \Gamma$ for all $m\geq 0$.

For each $i=1,\dots,r$, let $j_i\geq 0$ be maximal such that $Fa_i-a_i\in G(K^{p^{j_i}})$, if it exists, and set $j_i=\infty$ otherwise.
Let
$$T:=a_0+\sum_{i=1}^r\{F^na_i:n\geq-j_i\}$$
Note that if $j_i<\infty$ then $\{F^na_i:n\geq-j_i\}$ is the $F$-orbit $S(F^{-j_i}a_i;1)$.
On the other hand, if $j_i=\infty$ then $Fa_i-a_i$, and hence $F^na_i-a_i$ for all $n\in\mathbb Z$, is in the finite group $G(K\cap\Fq^{\alg})$.
Hence, in that case, $\{F^na_i:n\geq-j_i\}$ is finite.
So $T$ is a finite union of sets from $\mathcal S(G,F)$.
It is also clear that $E\subseteq T$.

We claim that $T\subseteq\overline{E}$ and that $T\subseteq\Gamma$.
We proceed by induction on $r\geq 1$.

Suppose $r=1$.
Then $E=a_0+S(a_1;1)$.
Hence the Zariski closed set $\overline E-a_0$ is $F$-invariant, and so also $F^{-1}$-invariant.
Since $a_1\in\overline E-a_0$ we have that $F^na_1\in \overline E-a_0$ for all $n\in\mathbb Z$.
In particular, $T=a_0+\{F^na_1:n\geq-j_1\}\subseteq\overline E$.

Next, still in the case of $r=1$, we show that $T\subseteq\Gamma$.
Since $E\subseteq\Gamma$, it suffices to show that $a_0+F^{-\ell}a_1\in\Gamma$ for $0<\ell\leq j_1$.
But we have that $Fa_1-a_1\in G(K^{p^{j_i}})$, so applying $F^{-1}$ repeatedly we get that $a_1-F^{-1}a_1,\dots,F^{-\ell+1}a_1-F^{-\ell}a_1\in G(K)$, and summing these telescopes to $a_1-F^{-\ell}a_1$.
Hence, $a_1-F^{-\ell}a_1\in G(K)$ and $F^{\ell}(a_1-F^{-\ell}a_1)=F^{\ell}a_1-a_1\in\Gamma$.
By $F$-purity of $\Gamma$ in $G(K)$, $a_1-F^{-\ell}a_1\in\Gamma$.
As $a_0+a_1\in E\subseteq\Gamma$, we thus have $a_0+F^{-\ell}a_1\in\Gamma$, as desired.

Assume now that $r>1$.
For each $m\geq 0$, consider
$$E_m:=a_0+F^ma_1+E(F^ma_2,\dots,F^ma_r;1).$$
So $\displaystyle E=\bigcup_{m\geq 0}E_m$.
Note that for each $i=2,\dots r$, if $j_i<\infty$ then
$$F(F^ma_i)-F^ma_i=F^m(Fa_i-a_i)\in G(K^{p^{j_i+m}})$$
and $m+j_i$ is maximal such.
If $j_i=\infty$ then $\displaystyle F(F^ma_i)-F^ma_i\in\bigcap_{j\geq 0}G(K^{p^j})$.
Hence the induction hypothesis yields that if we let
$$T_m:=a_0+F^ma_1+\sum_{i=2}^r\{F^{n}(F^ma_i):n\geq-j_i-m\}$$
then $T_m\subseteq\overline{E_m}\subseteq\overline E$ and $T_m\subseteq\Gamma$.
Note that
$$T_m=a_0+F^ma_1+\sum_{i=2}^r\{F^{n}a_i:n\geq-j_i\}$$
and $\displaystyle\bigcup_{m\geq 0}T_m\subseteq T$.
But we do not get all of $T$ in this way because only nonnegative $F$-iterates of $a_1$ appear in this union.

To show that $T\subseteq\overline E$ consider the Zariski closed set
$$Z:=\bigcap_{n_2\geq-j_2,\dots,n_r\geq -j_r}(\overline E-a_0-\sum_{i=2}^rF^{n_i}_ia_i).$$
Since $T_m\subseteq\overline E$ we have that $a_0+F^ma_1+ \sum_{i=2}^rF^{n_i}_ia_i \in\overline E$ for any fixed tuple of integers  $n_2\geq-j_2,\dots,n_r\geq -j_r$.
That is, $F^ma_1\in Z$ for all $m\geq 0$.
Hence $\overline{S(a_1;1)}\subseteq Z$.
But $\overline{S(a_1;1)}$ is an $F$-invariant, and hence $F^{-1}$-invariant, Zariski closed set.
So $F^\ell a_1\in Z$ for all integers $\ell $, in particular for all $\ell \geq-j_i$.
Hence $\displaystyle a_0+F^\ell a_1+\sum_{i=2}^r\{F^{n}a_i:n\geq-j_i\}\subseteq\overline E$ for all $\ell \geq-j_1$.
That is, $T\subseteq\overline E$.

Next, we show that $T\subseteq\Gamma$.
Since we already know that $T_m\subseteq\Gamma$ for all $m\geq 0$, it remains only to show that $\displaystyle a_0+F^{-\ell}a_1+\sum_{i=2}^r\{F^{n}a_i:n\geq-j_i\}\subseteq\Gamma$ for all $0<\ell\leq j_1$.
Fix integers $n_2\geq-j_2,\dots,n_r\geq -j_r$.
Note that $a_0+a_1+\sum_{i=2}^rF^{n_i}_ia_i\in T_0\subseteq\Gamma$.
But, from the argument in the $r=1$ case, using $F$-purity, we also know that $a_1-F^{-\ell}a_1\in\Gamma$.
Subtracting, we get $a_0+F^{-\ell}a_1+\sum_{i=2}^rF^{n_i}_ia_i\in\Gamma$, as desired.

We have completed the proof of the proposition in the case when all the $\delta_i=1$.
It remains to explain how to reduce to that case.
Let $\delta:=\operatorname{lcm}\{\delta_1,\dots,\delta_r\}$.
Then
$a_0+E(a_1,\dots,a_r;\delta_1,\dots,\delta_r)$ is equal to the union of
$$a_0+E(F^{m_1\delta_1}a_1,F^{m_1\delta_1+m_2\delta_2}a_2,\dots,F^{\sum_{i=1}^rm_i\delta_i}a_r;\delta)$$
as each $m_i$ ranges in the finite set $\{0,1,\dots,\frac{\delta}{\delta_i}\}$.
Hence it suffices to prove the proposition when all the $\delta_i=\delta$.
Finally, to deal with this case, noting that $\Gamma$ is an $F^\delta$-pure $\mathbb{Z}[F^\delta]$-submodule of $G(K)$, we can apply the already proved case of the proposition to $F^\delta$.
Since $\mathcal S(G,F^\delta)\subseteq \mathcal S(G,F)$, this suffices.
\end{proof}

\begin{remark}
\label{et-better}
The proof of Proposition~\ref{et} shows (under the same hypothesis) that if $E=a_0+E(a_1,\dots,a_r;1)\subseteq\Gamma$ then
$E\subseteq a_0+S(a_1,\dots,a_r;1)\subseteq\overline E\cap\Gamma$.
Indeed, clearly $E\subseteq a_0+S(a_1,\dots,a_r;1)\subseteq a_0+\sum_{i=1}^r\{F^na_i:n\geq-j_i\}=:T$, and it was shown in the proof that $T\subseteq\overline E\cap\Gamma$.
\end{remark}

\begin{corollary}
\label{iml-sparse}
Suppose $G$ is a commutative algebraic group over $\Fq$, $F:G\to G$ is the $q$-power Frobenius, $K$ is a function field extension of $\Fq$, $\Gamma\leq G(K)$ is an $F$-pure $\zf$-submodule, and $\Sigma\subseteq\Gamma$ is finite.
Suppose $\mathcal{L}\subseteq \Sigma^*$ is sparse.
Then there exists $T\subseteq\Gamma$, a finite union of elements of $\mathcal S(G,F)$, such that $[\mathcal L]_F\subseteq T\subseteq\overline{[\mathcal L]_F}$.
\end{corollary}

\begin{proof}
By~\cite[Proposition~7.1]{fsets-SML}, $\mathcal L$ is a finite union of simple sparse languages.
Apply the main clause of Lemma~\ref{expandsparse} to each of these with $M=\Gamma$ and $N=G$, noting that $F-\id$ is a surjective endomorphism of the algebraic group $G$ because it has finite kernel.
So $[\mathcal L]_F$ is a finite union of sets of the form $a_0+E(a_1,\dots,a_r;\delta_1,\dots,\delta_r)$, each lying in $\Gamma$ but with the $a_i$'s from $G$.
Now apply Proposition~\ref{et}.
\end{proof}

\medskip
\subsection{Intersecting with certain submodules}
In order to reduce to the $F$-pure setting, as required in the previous results, we will require the following general property of translates of sums of $F$-orbits.

\begin{lemma}
\label{intsub}
Suppose $(M,F)$ is an abelian group equipped with an injective endomorphism and $A\leq B\leq M$ are $\zf$-submodules with $F^{\ell}B\leq A$ for some $\ell\geq 0$.
If $S\subseteq B$ is from $\mathcal S(M,F)$ then $S\cap A$ is a finite union of sets from $\mathcal S(M,F)$.
\end{lemma}

\begin{proof}
Write $S=a_0+S(a_1,\dots,a_r;\delta_1,\dots,\delta_r)$.
Note, first of all, that as $S\subseteq B$, we have $F^{\delta_i}a_i-a_i\in B$ for all $i=1,\dots,r$.
Indeed,
$$F^{\delta_i}a_i-a_i= (a_0+a_1+\cdots+F^{\delta_i}a_i+\cdots a_r)-(a_0+a_1+\cdots+a_i+\cdots a_r).$$
It follows that $F^{n\delta_i}a_i-a_i\in B$, for all $n\geq 0$.
This in turn implies, using the fact that $F^\ell B\leq A$, that
\begin{equation}
\label{lmn}
F^{(\ell+m)\delta_i}a_i-F^{(\ell+n)\delta_i}a_i\ \in \ A
\end{equation}
for all $m\geq n\geq 0$ and $i=1,\dots,r$.

We now proceed to prove the lemma by induction on $r\geq 1$, dealing with the base case and induction step at the same time.

For each $j=1,\dots,r$ and $N> 0$, consider the subset of $S$ given by
\begin{eqnarray*}
S_{j,N}
&:=&
\{a_0+\sum_{i=1}^rF^{n_i\delta_i}a_i:n_1,\dots, n_r\geq 0\text{ and }n_j<N\}\\
&=&\bigcup_{m=0}^{N-1} a_0+F^{m\delta_j}a_j+S(a_1,\dots,a_{j-1},a_{j+1},\dots,a_r;\delta_1,\dots,\delta_{j-1},\delta_{j+1},\delta_r).
\end{eqnarray*}
If $r=1$ then $S_{j,N}$ is finite and hence so is $S_{j,N}\cap A$, and if $r>1$ then $S_{j,N}\cap A$ is a finite union of sets from $S(G,F)$ by induction.
Hence, if $S\cap A=S_{j,N}\cap A$ for some $j$ and $N$ then we are done.
We thus assume this never happens.
This means, in particular, that there is an element of $S\cap A$ of the form $a_0+\sum_{i=1}^r F^{n_i\delta_i}a_i$ for some $n_1,\dots,n_r\geq\ell$.
Using ~(\ref{lmn}) it follows that 
$$T:=a_0+S(F^{n_1\delta_1}a_1,\dots,F^{n_r\delta_r}a_r;\delta_1,\dots,\delta_r)\subseteq A.$$
Hence,
$\displaystyle S\cap A=T\cup\bigcup_{i=1}^r (S_{i,n_i}\cap A)$ which is a finite union of sets from $\mathcal S(G,F)$.
\end{proof}

\medskip
\subsection{Proof of Theorem~\ref{iml-module}}
Fix $(G,\Fq,F, X,\Gamma)$ as in the statement of the theorem.
For convenience, let us denote by $\mathcal F\subseteq\mathcal P(\Gamma)$ the collection of finite unions of sets of the form $S+\Lambda$ where $S\subseteq\Gamma$ is from $\mathcal S(G,F)$ and $\Lambda\leq\Gamma$ is an $\mathbb Z[F^r]$-submodule for some $r>0$.
We therefore wish to show that $X\cap\Gamma\in\mathcal F$

Clearly we may assume that $X$ is irreducible.
We proceed by induction on $d:=\dim X$.
The case of $d=0$ being clear, we assume $d>0$.

We start with a series of reductions.
Fix a function field extension $K$ of $\Fq$ such that $\Gamma\leq G(K)$ and $X$ is over~$K$.
By induction, we may assume that $X(K)$ is Zariski dense in $X$.

\begin{reduction}
\label{fpure}
We may assume that $\Gamma$ is $F$-pure in $G(K)$.
\end{reduction}

\begin{proof}
Let $\tilde\Gamma$ be the $F$-pure hull of $\Gamma$ in $G(K)$.
By Proposition~\ref{prop:commutative}, there is $r\geq 0$ such that $F^r\tilde\Gamma\leq\Gamma$.
In particular, $\tilde\Gamma$ is finitely generated as a $\zf$-submodule.
We wish to show that if the theorem holds for $\tilde\Gamma$ then it holds for $\Gamma$.
Suppose therefore that $\displaystyle X\cap\tilde\Gamma=\bigcup_{i=1}^\ell S_i+\Lambda_i$ where each $S_i\subseteq\tilde\Gamma$ is in $\mathcal S(G,F)$ and $\Lambda_i\leq\tilde\Gamma$ is a $\mathbb Z[F^{r_i}]$-submodule for some $r_i>0$.
By Theorem~\ref{spanning}, $\Lambda_i/F^r\Lambda_i$ is finite, and so in the above expression for $X\cap\tilde\Gamma$ we may replace the $\Lambda_i$ by $F^r\Lambda_i$.
That is, since $F^r\Lambda_i\leq\Gamma$, we may assume that all the $\Lambda_i\leq\Gamma$.
Hence
\begin{eqnarray*}
X\cap\Gamma
&=&
(X\cap\tilde\Gamma)\cap\Gamma\\
&=&
\bigcup_{i=1}^\ell (S_i+\Lambda_i)\cap\Gamma\\
&=&
\bigcup_{i=1}^\ell (S_i\cap\Gamma+\Lambda_i)
\end{eqnarray*}
and we are done by Lemma~\ref{intsub} applied to $A=\Gamma$, $B=\tilde\Gamma$, and $M=G$.
\end{proof}

Let $E$ be the largest connected algebraic subgroup of $\operatorname{Stab}(X)$ that is defined over $\mathbb F_p^{\alg}$.
(As the class of connected algebraic subgroups over $\mathbb F_p^{\alg}$ is preserved by summation, this exists.)

\begin{reduction}
\label{span1}
We may assume that $\Gamma$ admits an $F$-spanning set $\Sigma$ such that
$\mathcal G:=\{(v,w)\in(\Sigma\times \Sigma)^* : [v]_F- [w]_F\in E\}$ is regular.
\end{reduction}

\begin{proof}
Let $\mu:G\times G\to G$ be the morphism $\mu(g,h)=g-h$.
By Corollary~\ref{iml-fauto-module}, $(\Gamma\times\Gamma)\cap\mu^{-1}(E)$ is $F$-automatic.
Hence, there is $r>0$ and an $F^r$-spanning set $\Sigma$ for $\Gamma$ such that
$$\{(v,w)\in(\Sigma\times \Sigma)^* : [v]_{F^r}- [w]_{F^r}\in E\}$$
is regular.
Since $\mathcal S(G,F^r)\subseteq\mathcal S(G,F)$, and $\Gamma$ remains $F^r$-pure in $G(K)$, it suffices to prove the theorem for $(G,\mathbb F_{q^r},F^r, X,\Gamma)$.
\end{proof}

Fix, therefore, such an $F$-spanning set $\Sigma$ for $\Gamma$.
We know that $X\cap\Gamma$ is $F$-automatic by Corollary~\ref{iml-fauto-module}.
Hence
$$\mathcal L:=\{w\in\Sigma^*:[w]_F\in X\}$$
is a regular language.
We also know by Lemma~\ref{unamb} applied to the equivalence relation of being in the same coset modulo $E$, that there is a regular language $\mathcal L_0\subseteq\Sigma^*$ such that for every for every $w\in \Sigma^*$ there is a unique $v\in \mathcal L_0$ with $|v|\leq |w|$ and $[v]_F- [w]_F\in E$.
Let
$$\mathcal E:=\{v\in\mathcal L_00^*:\text{ there is }w\in\mathcal L\text{ such that }|v|=|w|\text{ and }[v]_F-[w]_F\in E\}.$$
Note that $\mathcal E$ is regular as it is the projection onto the first co-ordinate of the regular language $\mathcal (L_00^*\times\mathcal L)\cap\mathcal G$.
It is easily verified that
\begin{equation}
\label{lf}
[\mathcal L]_F=X\cap\Gamma,
\end{equation}
\begin{equation}
\label{el}
\mathcal E\subseteq\mathcal L,
\end{equation}
\begin{equation}
\label{le}
\text{for every $w\in \mathcal L$ there is $v\in \mathcal E$ with $[v]_F- [w]_F\in E$, and}
\end{equation}
\begin{equation}
\label{!}
\text{for every $v_1,v_2\in \mathcal E$, if $|v_1|=|v_2|$ and $[v_1]_F- [v_2]_F\in E$ then $v_1=v_2$.}
\end{equation}

\begin{reduction}
\label{toE}
It suffices to prove that $[\mathcal E]_F\subseteq T\subseteq X$ for some $T\in\mathcal F$.
\end{reduction}

\begin{proof}
Since $E$ is defined over a finite field, $\Lambda:=E\cap\Gamma$ is a $\mathbb Z[F^r]$-submodule of $\Gamma$ for some $r>0$.
By~(\ref{le}) we have that $[\mathcal L]_F\subseteq [\mathcal E]_F+\Lambda$.
Suppose we have proved that $[\mathcal E]_F\subseteq T\subseteq X$  for some $T\in\mathcal F$.
Then $[\mathcal L]_F\subseteq T+\Lambda\subseteq X+\Lambda$.
Since $\Lambda\subseteq\operatorname{Stab}(X)$, this proves that $[\mathcal L]_F\subseteq T+\Lambda\subseteq X$.
Since $T+\Lambda\subseteq\Gamma$, we get from~(\ref{lf}) that $X\cap\Gamma=T+\Lambda$.
As $T+\Lambda\in\mathcal F$, this suffices.
\end{proof}

Let us denote by $\mathcal A$ the automaton that we built in~$\S$\ref{sect-automaton} recognising $\mathcal L$, and by $Z_1,\dots, Z_n$ the states of $\mathcal A$, with $Z_1=X(K)$ the starting state.
Recall that each $Z_i$ is of the form $(X-\gamma)^{q^{-\ell}}(K)$ for some $\gamma\in \Gamma$ which has an $F$-expansion of length~$\ell$.
Hence $\dim\overline{Z_i}\leq d$, and if $\dim\overline{Z_i}=d$ then $\overline{Z_i}=(X-\gamma)^{q^{-\ell}}$.
Reindexing, let us assume that $Z_1,\dots, Z_m$ are $d$-dimensional for some $1\leq m\leq n$ and that the rest are of dimension strictly less than $d$.
Since $X(K)$ is Zariski dense in $X$, we have $\dim\overline{Z_1}=d$, and hence such $m$ exists.

Given $w\in \Sigma^*$, we define $\delta(w) = (X-[w]_F)^{q^{-|w|}}(K)\in \{Z_1,\ldots ,Z_n\}$, where $|w|$ is the length of $w$.
So $\delta(w)$ is the state the machine ends up in on input $w$.
We first observe that once $\mathcal A$ is in a state of low dimension then it stays there:

\begin{lemma}
\label{stay>m}
Suppose $w\in\Sigma^*$.
If $\delta(w)\in \{Z_{m+1},\ldots ,Z_n\}$ then, for all $v\in \Sigma^*$, $\delta(wv)\in \{Z_{m+1},\ldots ,Z_n\}$.
\end{lemma}

\begin{proof}
Indeed, 
$(X-[wv]_F)^{q^{-|wv|}}=\big((X-[w]_F)^{q^{-|w|}}-[v]_F\big)^{q^{-|v|}}$.
So if $Y$ is the Zariski closure of $\delta(w)$ in $(X-[w]_F)^{q^{-|w|}}$ then the Zariski closure of $\delta(wv)$ is contained in
$(Y-[v]_F)^{q^{-|v|}}$ whose dimension is at most $\dim Y$.
\end{proof}

We now decompose $\mathcal E$ into certain auxiliary sublanguages:
For each $k\leq m$ and $\ell=m+1,\dots,n,$ let
\begin{eqnarray*}
\mathcal E_k
&:=&
\{w\in\mathcal E:\delta(w)=Z_k\},\\
\mathcal O_\ell
&:=&
\{wa:w\in\mathcal E_k\text{ for some $k\leq m$ and $a\in\Sigma$ moves $\mathcal A$ moves from $Z_k$ to $Z_\ell$}\},\\
\mathcal N_\ell
&:=&
\{v\in\Sigma^*: uv\in\mathcal E\text{ for some }u\in\mathcal O_\ell\},\\
\mathcal M_\ell
&:=&
\{v\in\Sigma^*:[v]_F\in \overline{Z_\ell}\}.
\end{eqnarray*}
Lemma~\ref{stay>m} together with~(\ref{el}) implies
\begin{equation}
\label{decomp}
\mathcal E=\bigcup_{k=1}^m\mathcal E_k\cup\bigcup_{\ell=m+1}^n\mathcal O_\ell\mathcal N_\ell.
\end{equation}

\begin{lemma}
\label{njmj}
For each $\ell=m+1,\dots,n$, $\mathcal N_\ell\subseteq\mathcal M_\ell$ and $[\mathcal O_\ell\mathcal M_\ell]_F\subseteq X$.
\end{lemma}

\begin{proof}
Suppose $v\in\mathcal N_\ell$.
Then there is $wa\in\mathcal O_\ell$ such that $wav\in\mathcal E$.
Hence
\begin{eqnarray*}
[wav]_F\in X
&\implies&
[wa]_F+F^{|w|+1}([v]_F)\in X\\
&\implies&
F^{|w|+1}([v]_f)\in X-[wa]_F\\
&\implies&
[v]_F\in(X-[wa]_F)^{q^{-|w|-1}}(K)=Z_\ell.
\end{eqnarray*}
This proves that $\mathcal N_\ell\subseteq\mathcal M_\ell$.

Now, suppose $wa\in\mathcal O_\ell$ and $v\in\mathcal M_\ell$.
Then $z:=[v]_F\in\overline{Z_\ell}(K)=Z_\ell$.
(This last equality follows from the fact that $Z_\ell$ is the set of $K$-points of some variety.)
But as $wa$ moves $\mathcal A$ into the state $Z_\ell$ we have that
$Z_\ell=(X-[wa]_F)^{q^{-|w|-1}}(K)$.
Hence
$$[wav]_F=[wa]_F+F^{|w|+1}(z)\in X$$
as desired.
\end{proof}

Putting Lemma~\ref{njmj} together with~(\ref{decomp}) we get
\begin{equation}
\label{Fdecomp}
[\mathcal E]_F\ \subseteq\ \bigcup_{k=1}^m[\mathcal E_k]_F\cup\bigcup_{\ell=m+1}^n\mathcal [O_\ell\mathcal M_\ell]_F\ \subseteq\ X
\end{equation}

\begin{lemma}
\label{ksparse}
For each $k\leq m$, $\mathcal E_k$ is sparse.
\end{lemma}

\begin{proof}
It follows from the general theory of regularity that $\mathcal E_k$ is a regular sublanguage of $\mathcal E$.
Suppose $\mathcal E_k$ is not sparse.
Then, by one of the equivalent characterisations of sparsity given in~\cite[Proposition~7.1]{fsets-SML}, there exist $u,a,b,v\in\Sigma^*$ with $a$ and $b$ distinct nonempty words of the same length, such that
$u\{a,b\}^*v\subseteq\mathcal E_k$.

Let $w_1,w_2\in \{a,b\}^*$ be words of the same length.
Then
$$(X-[uw_1v])^{q^{-|uw_1v|}}(K)=Z_k=(X-[uw_2v])^{q^{-|uw_2v|}}(K).$$
Since $k\leq m$, taking Zariski closures yields
$$(X-[uw_1v])^{q^{-|uw_1v|}}=(X-[uw_2v])^{q^{-|uw_2v|}}.$$
Transforming by $F^{|uw_1v|}= F^{|uw_2v|}$ and translating by $F^{|uw_1|}([v])=F^{|uw_2|}([v])$, we get that $g(w_1,w_2):=[uw_1]-[uw_2]\in\operatorname{Stab}(X)$.
Letting $\lambda:=|a|$, note that
$$g(aw_1,aw_2)=([a]+F^\lambda([uw_1]))-([a]+F^\lambda([uw_2]))=F^\lambda g(w_1,w_2).$$
So $\{g(w_1,w_2):w_1,w_2\in \{a,b\}^*, |w_1|=|w_2|\}$ is preserved by $F^\lambda$.
Hence, so is the Zariski closure $H\leq\operatorname{Stab}(X)$ of the subgroup generated by this set.
It follows that~$H$, and hence its connected component $H_0$, is  over a finite field.
So $H_0\leq E$.

Let $N$ be the index of $H_0$ in $H$.
Fix $N+2$ distinct words $w_0,\dots,w_{N+1}\in\{a,b\}^*$ of the same length.
Then $\{g(w_0,w_j):j=1,\dots,N+1\}\subseteq H$ cannot all lie in distinct cosets of $H_0$.
So for some $i\neq j$ we have that
$$[uw_jv]-[uw_iv]=[uw_j]-[uw_i]=g(w_0,w_i)-g(w_0,w_j)\in H_0\leq E.$$
But as $uw_iv,uw_jv\in\mathcal E$ are distinct but of the same length, this contradicts~(\ref{!}).
\end{proof}

\begin{reduction}
\label{om}
It suffices to prove that for each sparse $\mathcal O\subseteq\Sigma^*$ and $U\in \mathcal F$, if
$$[\mathcal O]_F\star U:=\{[w]_F+F^{|w|}(\gamma):w\in\mathcal O, \gamma\in U\}$$
then $[\mathcal O]_F\star U\subseteq T\subseteq X$ for some $T\in\mathcal F$.
\end{reduction}

\begin{proof}
Since $\mathcal E_k$ is sparse for each $k\leq m$ by Lemma~\ref{ksparse}, and $\Gamma$ is $F$-pure in $G(K)$ by Reduction~\ref{fpure}, we may apply Corollary~\ref{iml-sparse} to see that each $\mathcal E_k\subseteq T\subseteq X$ for some $T\in\mathcal F$.
Hence, by~(\ref{Fdecomp}), and Reduction~\ref{toE}, it suffices to show that, for all $\ell+m+1,\dots,n$, $\mathcal [O_\ell\mathcal M_\ell]_F\subseteq T\subseteq X$ for some $T\in\mathcal F$.
But $[O_\ell\mathcal M_\ell]_F=[O_\ell]_F\star[\mathcal M_\ell]_F$ by definition, $\mathcal O_\ell$ is sparse by Lemma~\ref{ksparse} as it is contained in $\bigcup_{k=1}^m\mathcal E_k\Sigma$, and $[\mathcal M_\ell]_F=\overline{Z_\ell}\cap\Gamma\in\mathcal F$ by induction since $\ell>m$.
\end{proof}

Taking finite unions we may assume that $\mathcal O=u_1w_1^*u_2\cdots u_kw_k^*u_{k+1}$ is a simple sparse language and that $U=b+S+\Lambda$ where $S=S(b_1,\dots,b_t;\delta)\subseteq\Gamma$ for some $b,b_1,\dots,b_2\in G$ and $\Lambda\leq\Gamma$ is a $\mathbb Z[F^\delta]$-submodule for some $\delta>0$.
We may assume that $k>0$ as the result is trivial for $k=0$.
By Lemma~\ref{expandsparse} we have
$$[\mathcal O]_F=a_0+E(a_1,\dots,a_r;\delta_1,\dots,\delta_r)$$
for some $a_0,\dots,a_r\in G$ and $\delta_1,\dots,\delta_r>0$, and for any $v\in\Sigma^*$
$$[\mathcal Ov]_F=a_0+E(a_1,\dots,a_{r-1},a_r+F^\ell[v]_F;\delta_1,\dots,\delta_r)$$
where $\ell=\sum_{i=1}^{k+1}|u_i|$.
Hence
\begin{eqnarray*}
[\mathcal O]_F\star U
&=&
\bigcup_{\gamma\in U}[\mathcal Ov_\gamma]_F\ \ \text{ where }\gamma=[v_\gamma]_F\\
&=&
\bigcup_{\gamma\in S+\Lambda}a_0+E(a_1,\dots,a_{r-1},a_r+F^\ell(b+\gamma);\delta_1,\dots,\delta_r)\\
&=&
\bigcup_{\gamma\in S+\Lambda}a_0+E(a_1,\dots,a_{r-1},a_r+F^\ell b,F^\ell\gamma;\delta_1,\dots,\delta_r,\delta)
\end{eqnarray*}
where the final equality uses the fact that $S+\Lambda$ is $F^\delta$-invariant.
As pointed out in the proof of Proposition~\ref{et}, setting $\rho=\operatorname{lcm}\{\delta_1,\dots,\delta_r,\delta\}$, each
$$a_0+E(a_1,\dots,a_{r-1},a_r+F^\ell b,F^\ell\gamma;\delta_1,\dots,\delta_r,\delta)$$
can be written as the union of the sets
$$a_0+E(F^{m_1\delta_1}a_1,F^{m_1\delta_1+m_2\delta_2}a_2,\dots,F^{\sum_{i=1}^rm_i\delta_i}(a_r+F^\ell b), F^{\sum_{i=1}^rm_i\delta_i+m_{r+1}\delta+\ell}\gamma;\rho)$$
where $0\leq m_i\leq\frac{\rho}{\delta_i}$ and $0\leq m_{r+1}\leq\frac{\rho}{\delta}$.
For ease of notation, fix  $\mu:=(m_1,\dots,m_{r+1})$ and let
\begin{eqnarray*}
a_{\mu,j}
&:=&
F^{\sum_{i=1}^jm_i\delta_i}a_j\ \ \text{ for }i=1,\dots,r-1\\
a_{\mu,r}
&:=&
F^{\sum_{i=1}^rm_i\delta_i}(a_r+F^\ell b)\ \ \text{ and}\\
r_\mu
&:=&
\sum_{i=1}^rm_i\delta_i+m_{r+1}\delta+\ell.
\end{eqnarray*}
We thus have
\begin{eqnarray*}
[\mathcal O]_F\star U
&=&
\bigcup_{\gamma\in S+\Lambda}\bigcup_\mu a_0+E(a_{\mu,1},\dots,a_{\mu,r},F^{r_\mu}\gamma;\rho).
\end{eqnarray*}
It therefore suffices to prove that, for each of these finitely many $\mu$, the set 
$$E_\mu:=\bigcup_{\gamma\in S+\Lambda} a_0+E(a_{\mu,1},\dots,a_{\mu,r},F^{r_\mu}\gamma;\rho)$$
is contained in an element of $\mathcal F$ that is contained in $X$.
Now, by the $F$-purity of $\Gamma$ in $G(K)$ given by Reduction~\ref{fpure}, we can apply Proposition~\ref{et} to see that each
$$a_0+E(a_{\mu,1},\dots,a_{\mu,r},F^{r_\mu}\gamma;\rho)$$
is contained in an element of $\mathcal F$ that is contained in $X$.
In fact, as explained in  Remark~\ref{et-better}, the proof of that proposition shows that
$$a_0+E(a_{\mu,1},\dots,a_{\mu,r},F^{r_\mu}\gamma;\rho)\subseteq a_0+S(a_{\mu,1},\dots,a_{\mu,r},F^{r_\mu}\gamma;\rho)\subseteq X\cap\Gamma.$$
Hence
$E_\mu\subseteq T\subseteq X$
where $T\subseteq\Gamma$ is given by
\begin{eqnarray*}
T
&:=&
\bigcup_{\gamma\in S+\Lambda}a_0+S(a_{\mu,1},\dots,a_{\mu,r},F^{r_\mu}\gamma;\rho)\\
&=&
a_0+S(a_{\mu,1},\dots,a_{\mu,r};\rho)+\bigcup_{\gamma\in S+\Lambda}S(F^{r_\mu}\gamma;\rho)\\
&=&
a_0+S(a_{\mu,1},\dots,a_{\mu,r};\rho)+F^{r_\mu}(S+\Lambda)
\end{eqnarray*}
where the final equality uses that $S+\Lambda$, and hence $F^{r_\mu}(S+\Lambda)$, is $F^\rho$-invariant.
So $T\in\mathcal F$, as desired.

This completes the proof of Theorem~\ref{iml-module}.\qed

\bigskip

\appendix

\section{Trivialising multiplicative tori}

\begin{proposition}
\label{subsect-tori}
Suppose $M$ is a multiplicative torus over $\Fq$; that is, it is an algebraic group that is isomorphic to $\bG_m^d$ over $\Fq^{\alg}$.
Then there exists an isomorphism between $M$ and $\bG_m^d$ over $\mathbb F_{q^\ell}$ with $\ell\leq \max(2^{11} 11!, 2^d d!)$.
\end{proposition}

\begin{proof}
Write $M={\rm Spec}(T)$ where $T$ is a finitely generated (geometrically) integral $\mathbb{F}_q$-algebra.
Let $\ell$ be least such that there exists an isomorphism between $M$ and $\bG_m^d$ over $\mathbb F_{q^\ell}$.
We have an $\Fql$-algebra isomorphism
$\Phi:T\otimes_{\mathbb{F}_q} \mathbb{F}_{q^{\ell}}\to \mathbb{F}_{q^{\ell}}[t_1^{\pm 1},\ldots ,t_d^{\pm 1}]$.
Let $F:\Fql\to\Fql$ denote the $q$-power Frobenius automorphism, and $\tau:=\id\otimes F$ the induced $\Fq$-algebra automorphism of $T\otimes_{\mathbb{F}_q} \mathbb{F}_{q^{\ell}}$.
Let $\sigma:=\Phi\tau\Phi^{-1}$ be the corresponding $\Fq$-algebra automorphism of $\mathbb{F}_{q^{\ell}}[t_1^{\pm 1},\ldots ,t_d^{\pm 1}]$.
Note that both $\tau$ and $\sigma$ extend $F$ on $\Fql$ and are of order $\ell$.

For each $i=1,\dots,d$ we have that $\sigma(t_i)=\alpha_i t_1^{a_{i,1}}\cdots t_d^{a_{i,d}}$ for some $\alpha_i\in \Fql^*$ and some integers $a_{i,j}$ such that $A=(a_{i,j})\in\operatorname{GL}_d(\mathbb Z)$.
Note that $A^\ell=\id$.
We claim that $A$ is of order $\ell$.
This will suffice because the maximal order of a torsion element in $\operatorname{GL}_d(\mathbb Z)$, at least when $d>10$, is $2^dd!$ by a theorem of Feit~\cite{feit}.

Suppose, toward a contradiction, that the order of $A$ is $k<\ell$.
Then, for each $i=1,\dots,d$, $\sigma^k(t_i)=\alpha_i t_i$ for some $\alpha_i\in\Fql^*$.
On the other hand, $\sigma^\ell(t_i)=t_i$.
Hence, letting $m:=\frac{\ell}{k}$, we get that
$$1=\alpha_i\sigma^k(\alpha_i)\sigma^{2k}(\alpha_i)\cdots \sigma^{(m-1)k}(\alpha_i)=\alpha_i F^k(\alpha_i) F^{2k}(\alpha_i)\cdots F^{(m-1)k}(\alpha_i).$$  
That is, each $\alpha_i$ has norm 1 over $\mathbb F_{q^m}$, and so by Hilbert's Theorem 90, there exist $\lambda_i\in\Fql^*$ such that $\alpha_i=\frac{\lambda_i}{F^k\lambda_i}$.
Hence, after a change of variables $t_i\mapsto\lambda_i t_i$, which corresponds to modifying $\Phi$, we may assume that $\sigma^k(t_i)=t_i$ for all $i$.
It follows that the fixed subring of $\sigma^k$ is $\mathbb F_{q^m}[t_1^{\pm 1},\ldots ,t_d^{\pm 1}]$.
Of course the fixed subring of $\tau^k=\id\otimes F^k$ is $T\otimes_{\Fq}\mathbb F_{q^m}$, and $\Phi$ restricts to an isomorphism of the fixed subrings.
We thus have $T\otimes_{\Fq}\mathbb F_{q^m}$ isomorphic to $\mathbb F_{q^m}[t_1^{\pm 1},\ldots ,t_d^{\pm 1}]$, contradicting the minimality of $\ell$.
\end{proof}

\bigskip
\section{Finding a $K^q$-basis}
\label{subsect-compK}
\noindent
Suppose we are given a finite presentation of an $\mathbb{F}_q$-algebra $R$ and $K={\rm Frac}(R)$.
In this appendix we explain how to find a basis for $K$ over $K^q$.

We have a set of generators $x_1,\ldots ,x_m$ for $R$ as an $\mathbb{F}_q$-algebra.
Using Gr\"obner basis, after reordering our variables if necessary, we can find some $t\le m$ such that $x_1,\ldots ,x_t$ are algebraically independent over $\mathbb{F}_q$ and such that $x_j$ is algebraic over $L:=\mathbb{F}_q(x_1,\ldots ,x_t)$ for $t+1\le j\le m$.
We impose a degree lexicographic ordering on the variables $x_{t+1},\ldots ,x_m$ and use Gr\"obner bases to find a collection of monomials $a_1,\ldots ,a_r$ in $x_{t+1},\ldots ,x_m$ such that the following hold:
\begin{itemize}
\item for each $j=t+1,\dots,m$, some power of $x_j$ is in $\{a_1,\ldots ,a_r\}$;
\item the (finite) collection $\mathcal{M}$ of monomials in $x_{t+1},\ldots ,x_m$ that are not a multiple of some element of $a_1,\ldots ,a_r$ form a basis for $K$ over $L$;
\item each element in $\mathcal{M}$ is degree lexicographically less than some $a_i$;
\item for each $j=1,\dots,r$ we have an explicit expression
\begin{equation}
\label{eq:aj}
a_j = \sum_{b\in \mathcal{M}} \lambda_{j,b} b,
\end{equation}
with $\lambda_{j,b}\in L$.
\end{itemize}
Applying the $q$-Frobenius isomorphism to $\mathcal M$ we see that
$$\mathcal{A}:=\{u^q\colon  u\in \mathcal{M}\}$$
is a basis for $K^q$ over $L^q$.
Moroever, as a basis for $L$ over $L^q$ is given by $x_1^{i_1}\cdots x_t^{i_t}$ with $0\le i_1,\ldots ,i_t<q$, we have that
$$\mathcal{U}:=\{x_1^{i_1}\cdots x_t^{i_t}u\colon  0\le i_1,\ldots ,i_t<q, u\in \mathcal{M}\}$$
is a basis for $K$ over $L^q$.

Note that as $K^q$ contains $L^q$, the set $\mathcal U$ spans $K$ as a $K^q$-vector space.
We wish to refine this into a basis $\mathcal B$.
Let $u_1,\ldots ,u_r$ be an enumeration of~$\mathcal{U}$.  We begin at step $0$ with $\mathcal{B}=\emptyset$.  At step $i$, to determine whether $u_i$ is in our basis, we check whether $u_i$ is in the $L^q$-span of 
the set $\displaystyle \mathcal{B}_i:= \bigcup_{j<i} u_j\mathcal{A}$.
Notice that for each $a\in \mathcal{A}$ and each $j<i$, using the relations given by~(\ref{eq:aj}), we can explicitly express $u_ja$ as an $L^q$-linear combinations of elements of $\mathcal{U}$.
It then becomes a simple matter of linear algebra to determine whether $u_i$ is in the $L^q$-span of $\mathcal{B}_i$.
If it is not, we add $u_i$ to $\mathcal{B}$; otherwise, we leave $\mathcal{B}$ unchanged.  We claim that the set $\mathcal{B}$ is a basis for $K$ over $K^q$.
First, to see that it spans $K$, suppose toward a contradiction that some $u_i$ is not in the $K^q$-span of~$\mathcal B$, and assume that $i$ is least such.
Then $u_i\not\in \mathcal{B}$ and so by construction $u_i$ is in the $L^q$-span of the elements
$\bigcup_{j<i} u_j\mathcal{A}$.  But since $L^q\mathcal{A}\subseteq K^q$, we see that $u_i$ is in the $K^q$-span of $u_1,\ldots ,u_{i-1}$, which by minimality of $i$ is contained in the $K^q$-span of $\mathcal{B}$, our desired contradiction.
Now, to see linear independence, notice that if $\mathcal{B}$ is dependent then there is some $i$ for which $u_i\in \mathcal{B}$ and for which we have a relation
$\displaystyle u_i = \sum_{j<i} \lambda_j^q u_j$ with the $\lambda_j\in K$.  But now we can write each $\lambda_j^q$ as an $L^q$-linear combination of elements of $\mathcal{A}$, and this means that $u_i$ is in the $L^q$-span of the elements
$\bigcup_{j<i} u_j\mathcal{A}$, contradicting the fact that $u_i\in \mathcal{B}$.


\end{document}